\documentclass[10pt]{amsart}

\usepackage{amssymb}
\usepackage{bm}
\usepackage{graphicx}
\usepackage{psfrag}
\usepackage{hyperref}
\hypersetup{colorlinks=true, linkcolor=blue, citecolor=red, urlcolor=wine}
\usepackage{color}
\usepackage{url}
\usepackage{algpseudocode}
\usepackage{fancyhdr}
\usepackage{tikz-cd}
\usepackage{xy}
\input xy
\xyoption{all}
\usepackage{stmaryrd}
\usepackage{calrsfs}

\newtheorem*{maintheorem*}{Main Theorem}
\newtheorem{theorem}{Theorem}[section]
\newtheorem{prop}[theorem]{Proposition}

\newtheorem{conj}[theorem]{Conjecture}
\newtheorem{question}[theorem]{Question}
\newtheorem{lemma}[theorem]{Lemma}
\newtheorem{cor}[theorem]{Corollary}
\theoremstyle{definition}

\newtheorem{remark}[theorem]{Remark}
\newtheorem{example}[theorem]{Example}
\numberwithin{equation}{section}

\newcommand{\ii}{\mathcal{A}}
\newcommand{\nn}{\mathbb{N}}
\newcommand{\pp}{\mathbb{P}}
\newcommand{\qq}{\mathbb{Q}}
\newcommand{\rr}{\mathbb{R}}

\newcommand{\zz}{\mathbb{Z}}

\newcommand{\gp}{\text{gp}}

\newcommand{\supp}{\text{supp}}

\providecommand\ldb{\llbracket}
\providecommand\rdb{\rrbracket}

\newcommand{\rank}{\text{rank}}



\begin{document}

\title{On the additive structure of algebraic valuations of \\ polynomial semirings}

\author{Jyrko Correa-Morris}
\address{Department of Mathematics\\Miami Dade College\\Miami, FL 33135}
\email{jcorrea7@mdc.edu}

\author{Felix Gotti}
\address{Department of Mathematics\\MIT\\Cambridge, MA 02139}
\email{fgotti@mit.edu}

\date{\today}

\begin{abstract}
	In this paper, we study factorizations in the additive monoids of positive algebraic valuations $\nn_0[\alpha]$ of the semiring of polynomials $\nn_0[X]$ using a methodology introduced by D. D. Anderson, D. F. Anderson, and M. Zafrullah in 1990. A cancellative commutative monoid is atomic if every non-invertible element factors into irreducibles. We begin by determining when $\nn_0[\alpha]$ is atomic, and we give an explicit description of its set of irreducibles. An atomic monoid is a finite factorization monoid (FFM) if every element has only finitely many factorizations (up to order and associates), and it is a bounded factorization monoid (BFM) if for every element there is a bound for the number of irreducibles (counting repetitions) in each of its factorizations. We show that, for the monoid $\nn_0[\alpha]$, the property of being a BFM and the property of being an FFM are equivalent to the ascending chain condition on principal ideals (ACCP). Finally, we give various characterizations for $\nn_0[\alpha]$ to be a unique factorization monoid (UFM), two of them in terms of the minimal polynomial of $\alpha$. The properties of being finitely generated, half-factorial, and length-factorial are also investigated along the way.
\end{abstract}

\maketitle

\bigskip
\section{Introduction}
\label{sec:intro}

The study of the deviation of rings of integers from being UFMs in connection to their divisor class groups earned significant attention in the 1960s with the influence of the number theorists L.~Carlitz~\cite{lC60} and W.~Narkiewicz~\cite{wN64,wN66}. Much of the divisibility theory of rings of integers, including their divisor class groups, carries over to Dedekind domains and, more generally, Krull domains. Motivated by this fact, the phenomenon of non-uniqueness of factorizations in the contexts of Dedekind and Krull domains was later investigated by A. Zaks in~\cite{aZ76} and~\cite{aZ80}, respectively. Since then, techniques to study factorizations in the more general context of cancellative commutative monoids, specially in Krull monoids~\cite{GGSS10} and multiplicative monoids of integral domains~\cite{AAZ90}, have been systematically developed, giving rise to what we know today as factorization theory.
\smallskip

In this paper, we are primarily concerned with factorizations in monoids of the form $(\nn_0[\alpha],+)$, where $\nn_0[\alpha] = \{f(\alpha) : f(X) \in \nn_0[X]\}$ is the homomorphic image of the semiring of polynomials $\nn_0[X]$ that we obtain after evaluating at $\alpha \in \rr \setminus \{0\}$. These monoid valuations are cancellative and commutative. The class of cancellative commutative monoids is the most natural abstraction of the class consisting of multiplicative monoids of integral domains. In addition, this class is important from the factorization-theoretical perspective because it is the most suitable abstract framework to formally define the notion of a factorization, as observed by F. Halter-Koch in~\cite{fHK92}. From now on, every monoid we mention here is tacitly assumed to be cancellative and commutative.
\smallskip

Following P.~M. Cohn~\cite{pC68}, we say that a monoid is atomic if every non-invertible element factors into irreducibles. A monoid satisfies the ascending chain condition on principal ideals (ACCP) if every increasing sequence of ideals eventually stabilizes. It follows immediately that every monoid satisfying the ACCP is atomic. There are integral domains whose multiplicative monoids are atomic but do not satisfy the ACCP; the first example was constructed by A. Grams~\cite{aG74}. An atomic monoid is called a finite factorization monoid (FFM) if every element admits only finitely many factorizations, and it is called a bounded factorization monoid (BFM) if for every element there is a bound for the number of irreducibles (counting repetitions) in each of its factorizations. These notions were introduced by D.~D. Anderson, D.~F. Anderson, and M. Zafrullah~\cite{AAZ90}  in the context of the diagram of Figure~\ref{fig:AAZ's atomic chain for monoids} to carry out the first systematic study of factorizations in integral domains. Following Zaks~\cite{aZ76}, we call an atomic monoid half-factorial (HFM) if any two factorizations of the same element have the same number of irreducibles (counting repetitions). The implications in the diagram shown in Figure~\ref{fig:AAZ's atomic chain for monoids} hold for any monoid. For the sake of consistency, in the same diagram, we let the (nonstandard) acronym ATM stand for the term `atomic monoid'.
\smallskip

\begin{figure}[h]
	\begin{tikzcd}
		\textbf{ UFM } \ \arrow[r, Rightarrow] \arrow[red, r, Leftarrow, "/"{anchor=center,sloped}, shift left=1.7ex] \arrow[d, Rightarrow, shift right=1ex] \arrow[red, d, Leftarrow, "/"{anchor=center,sloped}, shift left=1ex]& \ \textbf{ HFM } \arrow[d, Rightarrow, shift right=0.6ex] \arrow[red, d, Leftarrow, "/"{anchor=center,sloped}, shift left=1.3ex] \\
		\textbf{ FFM } \ \arrow[r, Rightarrow]	\arrow[red, r, Leftarrow, "/"{anchor=center,sloped}, shift left=1.7ex] & \ \textbf{ BFM } \arrow[r, Rightarrow]  \arrow[red, r, Leftarrow, "/"{anchor=center,sloped}, shift left=1.7ex] & \textbf{ ACCP monoid}  \arrow[r, Rightarrow] \arrow[red, r, Leftarrow, "/"{anchor=center,sloped}, shift left=1.7ex]  & \textbf{ATM}
	\end{tikzcd}
	\caption{The implications in the diagram show the known inclusions among the subclasses of atomic monoids we have previously mentioned. The diagram also emphasizes (with red marked arrows) that none of the shown implications is reversible.}
	\label{fig:AAZ's atomic chain for monoids}
\end{figure}
\smallskip

The study of factorization theory of monoids stemming from the semiring $\nn_0[X]$ and its valuations has been the subject of several recent papers. Methods to factorize polynomials in $\nn_0[X]$ were studied in~\cite{hB13}. In addition, a more systematic investigation of factorizations in the multiplicative monoid of the semiring $\nn_0[X]$ was recently carried out by F. Campanini and A. Facchini~in \cite{CF19}, where the similarity between the structure of $\nn_0[X]$ and that of Krull monoids was highlighted. On the other hand, it was proved by P. Cesarz et al.~\cite{CCMS09} that for reasonable quadratic algebraic integers~$\alpha$, the multiplicative monoid of the valuation semiring $\nn_0[\alpha]$ has full infinite elasticity (the elasticity is a factorization invariant introduced in~\cite{rV90}, and it has been extensively studied since then). The additive monoids of rational valuations $\nn_0[q]$ of the semiring $\nn_0[X]$ have also been investigated by Chapman et al. in~\cite{CGG20}, where it was proved, among other results, that the lengths of all factorizations of each element of $\nn_0[q]$ form an arithmetic progression. A similar result for more general additive submonoids of $\qq$ was recently established in~\cite{hP20}.
\smallskip

In this paper, we study atomicity and factorizations of the additive monoids $\nn_0[\alpha]$, where $\alpha$ is a positive real number. Such additive monoids are called here monoid valuations of $\nn_0[X]$. The class of monoid valuations of $\nn_0[X]$ includes, as special cases, $\nn_0[X]$ (the valuation at any positive transcendental number) and all monoid valuations $\nn_0[q]$ with $q$ being a positive rational, called here rational monoid valuations. Rational monoid valuations are Puiseux monoids, and the atomicity of the latter has been recently studied in connection to monoid rings~\cite{CG19} and upper triangular matrices over information semialgebras~\cite{BG20} (see also \cite{GGT19} and references therein). The additive monoids of transcendental valuations of $\nn_0[X]$ are free and, therefore, trivial from the factorization-theoretical perspective. Therefore we focus on monoids $\nn_0[\alpha]$, where $\alpha$ is a positive algebraic number. We call these monoids algebraic monoid valuations. Although algebraic monoid valuations of $\nn_0[X]$ are natural generalizations of rational monoid valuations (studied in~\cite{CGG20}), most of the arithmetic and factorization properties of rational monoid valuations do not hold or trivially generalize to algebraic monoid valuations. This is because, as we shall see later, most of such properties for an algebraic monoid valuation $\nn_0[\alpha]$ depend on the minimal polynomial of $\alpha$.
\smallskip

With our study we accomplish two goals. First, we refine the diagram in Figure~\ref{fig:AAZ's atomic chain for monoids} for the class of monoid algebraic valuations $\nn_0[\alpha]$. We show that for monoids in this class, being a UFM and being an HFM are equivalent conditions, and any of these conditions holds precisely when the degree of the minimal polynomial of $\alpha$ coincides with the number of atoms of $\nn_0[\alpha]$ (Theorem~\ref{thm:UFM characterization}). Also, we prove that being an FFM, being a BFM, and satisfying the ACCP are equivalent conditions for any monoid $\nn_0[\alpha]$ (Theorem~\ref{thm:BFM/FFM equivalence}). In addition, we provide examples of monoids to verify that no other implication in the diagram shown in Figure~\ref{fig:AAZ's atomic chain for monoids} becomes an equivalence in the class of monoid valuations of $\nn_0[X]$. Our second goal here is to investigate the properties of being finitely generated and being length-factorial in the class of algebraic monoid valuations, and then to study how these two properties fit in the diagram shown in Figure~\ref{fig:AAZ's atomic chain for monoids}. This second goal is achieved as indicated in Figure~\ref{fig:atomic classes of CASs}. Following Chapman et al.~\cite{CCGS21}, we say that an atomic monoid is a length-factorial monoid (LFM) if different factorizations of the same element have different numbers of irreducibles (counting repetitions). Length-factoriality was first studied by J. Coykendall and W.~W. Smith~\cite{CS11} in the context of integral domains under the term `other-half-factoriality'. For any positive algebraic number $\alpha$, we prove that $\nn_0[\alpha]$ is an LFM but not UFM if and only if the degree of the minimal polynomial of $\alpha$ precedes in $\zz$ the number of atoms of $\nn_0[\alpha]$ (Theorem~\ref{thm:OHFM characterization}). In the diagrams illustrated in Figure~\ref{fig:atomic classes of CASs}, we let the acronym FGM stand for the term `finitely generated monoid'.
\smallskip

\begin{figure}[h]
	\begin{tikzcd}[cramped]
		\big[ \textbf{UFM} \arrow[r, Leftrightarrow] & \textbf{HFM} \big] \arrow[r, Rightarrow] \arrow[red, r, Leftarrow, "/"{anchor=center,sloped}, shift left=1.5ex] & \textbf{LFM} \arrow[blue, d, Rightarrow, shift right=0.8ex] \arrow[red, d, Leftarrow, "/"{anchor=center,sloped}, shift left=1.6ex] \\
			&	& \textbf{FGM} \arrow[r, Rightarrow] \arrow[red, r, Leftarrow, "/"{anchor=center,sloped}, shift left=1.5ex] & \big[ \textbf{FFM} \arrow[r, Leftrightarrow] & \textbf{BFM} \arrow[r, Leftrightarrow] & \textbf{ACCP monoid} \big] \arrow[r, Rightarrow] \arrow[red, r, Leftarrow, "/"{anchor=center,sloped}, shift left=1.5ex] & \textbf{ATM}
	\end{tikzcd}
	\caption{The diagram shows, as its non-obvious implications, the main results we establish in this paper, where the vertical implication (in blue) holds for all algebraic monoid valuations of $\nn_0[X]$ while the rest of the implications hold for all monoid valuations of $\nn_0[X]$. The diagram also emphasizes (with red marked arrows) the implications that are not reversible: we construct here classes of algebraic monoid valuations of $\nn_0[X]$ witnessing the failure of such reverse implications.}
	\label{fig:atomic classes of CASs}
\end{figure}
\smallskip

\bigskip
\section{Notation and Background}
\label{sec:background}

We let $\mathbb{N}$ and $\nn_0 := \nn \cup \{0\}$ denote the set of positive and nonnegative integers, respectively, and we let $\pp$ denote the set of primes. In addition, for $X \subseteq \rr$ and $\beta \in \rr$, we set $X_{\ge \beta} := \{x \in X : x \ge \beta\}$; in a similar way, we use the notations $X_{\le \beta}$, $X_{> \beta}$, and $X_{< \beta}$. For $q \in \qq_{> 0}$, we denote the unique $n,d \in \nn$ such that $q = \frac nd$ and $\gcd(n,d) = 1$ by $\mathsf{n}(q)$ and $\mathsf{d}(q)$, respectively.

\smallskip
\subsection{Atomic Monoids}

We tacitly assume that all monoids in this paper are cancellative, commutative, and unless we specify otherwise, additively written. In addition, every monoid we treat here is either a group or a \emph{reduced} monoid, that is, its only invertible element is the identity element. Let $M$ be a reduced monoid. For $S \subseteq M$, we let $\langle S \rangle$ denote the submonoid of $M$ generated by $S$. If there exists a finite subset $S$ of $M$ such that $M = \langle S \rangle$, then $M$ is said to be a \emph{finitely generated monoid} or an \emph{FGM}. An element $a \in M \setminus \{0\}$ is an \emph{atom} provided that for all $x,y \in M$ the equality $a = x+y$ implies that either $x=0$ or $y=0$. The set of atoms of $M$ is denoted by $\ii(M)$, and $M$ is \emph{atomic} if $M = \langle \ii(M) \rangle$. On the other hand, $M$ is \emph{antimatter} if $\mathcal{A}(M)$ is empty. Finitely generated monoids are atomic.

A subset $I$ of $M$ is an \emph{ideal} of $M$ if $I + M \subseteq I$, and an ideal $I$ is \emph{principal} if $I = x + M$ for some $x \in M$. If $y \in x + M$, then we say that $x$ \emph{divides} $y$ \emph{in} $M$ and write $x \mid_{M} y$. An element $p \in M \setminus \{0\}$ is a \emph{prime} provided that for all $x,y \in M$ satisfying $p \mid_M x+y$ either $p \mid_M x$ or $p \mid_M y$. Clearly, every prime is an atom. The monoid $M$ satisfies the \emph{ascending chain condition on principal ideals} or the \emph{ACCP}) if each increasing sequence of principal ideals of $M$ eventually stabilizes. If a monoid satisfies the ACCP, then it is atomic \cite[Proposition~1.1.4]{GH06}. Atomic monoids may not satisfy the ACCP, as we shall see in Proposition~\ref{prop:infinitely many atomic CARS without the ACCP}.

The \emph{difference group} of $M$, denoted here by $\gp(M)$, is the abelian group (unique up to isomorphism) satisfying that any abelian group containing a homomorphic image of~$M$ also must contain a homomorphic image of $\gp(M)$. The monoid $M$ is \emph{torsion-free} provided that $\gp(M)$ is a torsion-free abelian group. The monoids we are interested in this paper are torsion-free. On the other hand, the \emph{rank} of $M$, denoted by $\rank \, M$, is the rank of the $\zz$-module $\gp(M)$, or equivalently, the dimension of the $\qq$-vector space $\qq \otimes_\zz \gp(M)$. Clearly, $\qq \otimes_\zz \gp(M)$ contains an additive copy of the monoid $M$ via the embedding $M \hookrightarrow \gp(M) \hookrightarrow \qq \otimes_\zz \gp(M)$, where the last map is injective because $\qq$ is a flat $\zz$-module.

For the rest of this section, assume that $M$ is atomic. The free (commutative) monoid on $\ii(M)$ is denoted by $\mathsf{Z}(M)$. An element $z = a_1 + \cdots + a_\ell \in \mathsf{Z}(M)$, where $a_1, \dots, a_\ell \in \ii(M)$, is a \emph{factorization} in $M$ of \emph{length} $|z| := \ell$. As $\mathsf{Z}(M)$ is free, there exists a unique monoid homomorphism $\pi_M \colon \mathsf{Z}(M) \to M$ satisfying $\pi_M(a) = a$ for all $a \in \ii(M)$. When there seems to be no risk of ambiguity, we write $\pi$ instead of $\pi_M$. For each $x \in M$, we set
\[
	\mathsf{Z}(x) := \pi^{-1}(x) \subseteq \mathsf{Z}(M) \quad \text{ and } \quad \mathsf{L}(x) := \{|z| : z \in \mathsf{Z}(x)\}.
\]
Since $M$ is atomic, the sets $\mathsf{Z}(x)$ and $\mathsf{L}(x)$ are nonempty for all $x \in M$. If $|\mathsf{Z}(x)| < \infty$  (resp., $|\mathsf{L}(x)| < \infty$) for all $x \in M$, then $M$ is a \emph{finite factorization monoid} or an \emph{FFM} (resp., a \emph{bounded factorization monoid} or a \emph{BFM}). Clearly, every FFM is a BFM. It follows from \cite[Proposition~2.7.8]{GH06} that every FGM is an FFM, and it follows from \cite[Corollary~1.3.3]{GH06} that every BFM satisfies the ACCP. The bounded and finite factorization properties in the setting of integral domains were recently surveyed by D. F. Anderson and the second author in~\cite{AG21}. If $|\mathsf{Z}(x)| = 1$ (resp., $|\mathsf{L}(x)| = 1$) for all $x \in M$, then $M$ is a \emph{unique factorization monoid} or a \emph{UFM} (resp., a \emph{half-factorial monoid} or an \emph{HFM}). It is clear from the definitions that every UFM is both an FFM and an HFM. Finally, $M$ is a \emph{length-factorial monoid} or an \emph{LFM} provided that for all $x \in M$ and $z_1, z_2 \in \mathsf{Z}(x)$, the equality $|z_1| = |z_2|$ implies that $z_1 = z_2$. See the recent survey~\cite{GZ20} by A. Geroldinger and Q. Zhong for more background in factorizations in atomic monoids. A summary of the implications mentioned in this subsection is illustrated in the diagram of Figure~\ref{fig:AAZ's atomic chain for monoids including LFM/FGM}, which is an enhanced version of that of Figure~\ref{fig:AAZ's atomic chain for monoids}.
\smallskip

\begin{figure}[h] 
	\begin{tikzcd}
		\textbf{ATM }  \arrow[r, Leftarrow] \arrow[red, r, Rightarrow, "/"{anchor=center,sloped}, shift left=1.7ex] & \textbf{ LFM }  \arrow[r, Leftarrow] \arrow[red, r, Rightarrow, "/"{anchor=center,sloped}, shift left=1.7ex] & \textbf{ UFM } \  \arrow[r, Rightarrow] \arrow[red, r, Leftarrow, "/"{anchor=center,sloped}, shift left=1.7ex]  \arrow[d, Rightarrow, shift right=1ex] \arrow[red, d, Leftarrow, "/"{anchor=center,sloped}, shift left=1ex] & \ \textbf{ HFM }  \arrow[d, Rightarrow, shift right=1ex] \arrow[red, d, Leftarrow, "/"{anchor=center,sloped}, shift left=1ex] \\
		&\textbf{FGM}  \arrow[r, Rightarrow] \arrow[red, r, Leftarrow, "/"{anchor=center,sloped}, shift left=1.7ex] &\textbf{ FFM } \   \arrow[r, Rightarrow] \arrow[red, r, Leftarrow, "/"{anchor=center,sloped}, shift left=1.7ex] & \ \textbf{ BFM }  \arrow[r, Rightarrow] \arrow[red, r, Leftarrow, "/"{anchor=center,sloped}, shift left=1.7ex] & \textbf{ ACCP monoid}   \arrow[r, Rightarrow] \arrow[red, r, Leftarrow, "/"{anchor=center,sloped}, shift left=1.7ex] & \textbf{ ATM}
	\end{tikzcd}
	\caption{The implications in the diagram show the inclusions among the subclasses of atomic monoids we have introduced in this subsection. The diagram also emphasizes (with red marked arrows) that none of the shown implications is reversible.}
	\label{fig:AAZ's atomic chain for monoids including LFM/FGM}
\end{figure}
\smallskip

\smallskip
\subsection{Valuation Monoids and Semirings of $\nn_0[X]$}

As we have assumed for monoids, every semiring we deal with here is cancellative and commutative. Accordingly, we say that a triple $(S,+, \cdot)$, where $(S,+)$ and $(S \setminus \{0\}, \cdot)$ are monoids, is a \emph{semiring} if the multiplicative operation distributes over the additive operation, and the equality $0 \cdot x = 0$ holds for every $x \in S$. As it is customary, we shall denote a semiring $(S,+, \cdot)$ simply by $S$. The monoid $(S,+)$ is the \emph{additive monoid} of $S$.

All the semirings we consider in this paper are real homomorphic images of the semiring of polynomials $\nn_0[X]$. For $\alpha \in \rr \setminus \{0\}$, we call the semiring $\{f(\alpha) : f(X) \in \nn_0[X]\}$ the \emph{semiring valuation} of $\nn_0[X]$ at~$\alpha$ or simply a \emph{semiring valuation} of $\nn_0[X]$. It is clear that the semiring valuation of $\nn_0[X]$ at $\alpha$ is the subsemiring of $\rr$ generated by $\alpha$. Unless we explicitly state otherwise, from now on we reserve the notation $\nn_0[\alpha]$ for the additive monoid of the semiring valuation of $\nn_0[X]$ at~$\alpha$, which we call the \emph{monoid valuation} of $\nn_0[X]$ at $\alpha$ or simply a \emph{monoid valuation} of $\nn_0[X]$. When $\alpha$ is transcendental, the algebraic and arithmetic properties of $\nn_0[\alpha]$ are similar to those of the polynomial semiring $\nn_0[X]$. We will summarize them in the following proposition. For the sake of completeness, we include the proof of the next proposition; however, it is fair to remark that all the assertions of the proposition should be somehow known and are not too surprising in any case.

\begin{prop} \label{prop:transcendental case}
	Let $\tau \in \rr$ be a transcendental element. Then the following statements hold.
	\begin{enumerate}
		\item $\nn_0[\tau] = \bigoplus_{n \in \nn_0} \nn_0 \tau^n \cong \bigoplus_{n \in \nn} \nn_0$. 
		\smallskip
		
		\item $\nn_0[\tau]$ is a UFM and so an atomic monoid. Also, $\mathcal{A}(\nn_0[\tau]) = \{\tau^n : n \in \nn_0\}$.
		\smallskip
		
		\item $\emph{gp}(\nn_0[\tau]) \cong \bigoplus_{n \in \nn_0} \zz \tau^n$ and $\emph{\rank} \, \nn_0[\tau] = \aleph_0$.
		\smallskip
		
		\item $\nn_0[\tau] \cong \nn_0[X]$ as semirings. In particular, any two positive transcendental numbers give isomorphic semiring valuations of $\nn_0[X]$.
	\end{enumerate}
\end{prop}

\begin{proof}
	(1) It is clear that $\nn_0[\tau]$ is generated by $\{\tau^n : n \in \nn_0\}$ as a monoid, and the fact that $\tau$ is transcendental immediately implies that this set of generators is integrally independent. Therefore $\nn_0[\tau] = \bigoplus_{n \in \nn_0} \nn_0 \tau^n$, which is the free commutative monoid of rank $\aleph_0$.
	\smallskip
	
	(2) This is an immediate consequence of part~(1).
	\smallskip
	
	(3) Because the set $\{\tau^n : n \in \nn_0\}$ is integrally independent, $\gp(\nn_0[\tau]) \cong \bigoplus_{n \in \nn_0} \zz \tau^n$. In addition, the fact that $\{\tau^n : n \in \nn_0\}$ is a basis for the $\zz$-module $\gp(\nn_0[\tau])$ guarantees that $\rank \, \nn_0[\tau] = \aleph_0$.
	\smallskip
	
	(4) The ring homomorphism $\varphi \colon \zz[X] \to \rr$ consisting in evaluating at $\tau$ has trivial kernel because~$\tau$ is transcendental. Since $\varphi(\nn_0[X]) = \nn_0[\tau]$, the restriction of $\varphi$ to $\nn_0[X]$ is a semiring isomorphism between $\nn_0[X]$ and $\nn_0[\tau]$. The last statement follows immediately.
\end{proof}

When $\alpha$ is a nonzero algebraic number (resp., a nonzero rational number), we call $\nn_0[\alpha]$ an \emph{algebraic monoid valuation} (resp., a \emph{rational monoid valuation}) of $\nn_0[X]$.\footnote{The rational semiring valuations of $\nn_0[X]$ were recently investigated in \cite{CGG20} under the term `cyclic rational semirings'.} In light of Proposition~\ref{prop:transcendental case}, from now on we will primarily focus on algebraic monoid valuations of $\nn_0[X]$. Unlike the case when $\alpha$ is transcendental, when $\alpha$ is algebraic the monoid $\nn_0[\alpha]$ may not be a UFM. This is illustrated in the following examples.

\begin{example} \hfill \label{ex:two initial examples}
	\begin{enumerate}
		\item Consider the algebraic monoid valuation $\nn_0[\alpha]$, where $\alpha$ is the algebraic number $\alpha = 2 - \sqrt{2}$. The minimal polynomial of $\alpha$ is $m_\alpha(X) = X^2 - 4X + 2$. From the fact that $\alpha \in (0,1)$, it is not hard to show that $\mathcal{A}(\nn_0[\alpha]) = \{\alpha^n : n \in \nn_0\}$ (this is a special case of Theorem~\ref{thm:atomic characterization}). Now the identity $\alpha^2 - 4\alpha + 2 = 0$ ensures that $z_1 := 4\alpha$ and $z_2 := 2 + \alpha^2$ are two distinct factorizations of the same element in $\nn_0[\alpha]$, and so $\nn_0[\alpha]$ cannot be a UFM. Indeed, since $z_1$ and $z_2$ have different lengths, $\nn_0[\alpha]$ is not even an HFM.
		\smallskip
		
		\item Fix $q \in \qq \setminus \zz$ with $q > 1$, and consider the rational monoid valuation $\nn_0[q]$. It follows from \cite[Theorem~6.2]{GG18} that $\nn_0[q]$ is atomic with $\mathcal{A}(\nn_0[q]) = \{q^n : n \in \nn_0\}$. Indeed, since the generating sequence $(q^n)_{n \in \nn_0}$ is increasing, \cite[Theorem~5.6]{fG19} guarantees that $\nn_0[q]$ is an FFM. On the other hand, we observe that $z_1 := \mathsf{n}(q)$ and $z_2 := \mathsf{d}(q)q$ are two factorizations of the same element in $\nn_0[q]$, namely, $\mathsf{n}(q)$. Finally, the fact that $z_1$ and $z_2$ have different lengths implies that $\nn_0[q]$ is not an HFM. In particular, $\nn_0[q]$ is not a UFM.
	\end{enumerate}
\end{example} 

Assume that $\alpha \in \mathbb{C}$ is algebraic over $\qq$. It should not come as a surprise that in our study of an algebraic monoid valuation $\nn_0[\alpha]$, the minimal polynomial $m_{\alpha}(X) \in \qq[X]$ of $\alpha$ plays a crucial role. We call the set of exponents of the monomials appearing in the canonical representation of a polynomial $f(X) \in \qq[X]$ the \emph{support} of $f(X)$, and we denote it by $\supp \, f(X)$, that is,
\[
	\supp \, f(X) := \{n \in \nn_0 : f^{(n)}(0) \neq 0 \},
\]
where $f^{(n)}$ denotes the $n$-th derivative of $f$. Clearly, there is a unique $\ell \in \nn$ such that the polynomial $\ell m_\alpha(X) \in \zz[X]$ has content~$1$ (that is, the greatest common divisor of all the coefficients of $\ell m_\alpha(X)$ is~$1$). In addition, there exist unique polynomials $p_\alpha(X), q_\alpha(X) \in \nn_0[X]$ with $\ell m_\alpha(X) = p_\alpha(X) - q_\alpha(X)$ and $\supp \, p_\alpha(X) \bigcap \supp \, q_\alpha(X) = \emptyset$. We call the pair $(p_\alpha(X), q_\alpha(X))$ the \emph{minimal pair} of $\alpha$. Finally, we recall that the conjugates of  $\alpha$ (over $\qq$) are the complex roots of $m_\alpha(X)$.

Descartes' Rule of Signs states that the number of variations of sign of a polynomial $f(X) \in \rr[X]$ has the same parity as and is at least the number of positive roots of $f(X)$ (counting multiplicity). In addition, it was proved by D. R. Curtiss \cite{dC18} that there exists a polynomial $\phi(X) \in \rr[X]$, which he called a Cartesian multiplier, such that the number of variations of sign of $\phi(X)f(X)$ equals the number of positive roots of $f(X)$ (counting multiplicity). Using the density of $\qq$ in the real line, one can take such Cartesian multipliers in $\zz[X]$. We proceed to record Curtiss' result for future reference.

\begin{theorem} \cite[Section~5]{dC18} \label{thm:Curtiss' multipliers}
	For each $f(X) \in \rr[X]$, there exists $\phi(X) \in \zz[X]$ such that the number of variations of sign of $\phi(X)f(X)$ equals the number of positive roots of $f(X)$.
\end{theorem}

\bigskip
\section{Algebraic Considerations}

In this section, we address two algebraic aspects for semiring valuations of $\nn_0[X]$: we determine their additive rank, and we find necessary and sufficient conditions for two such semiring valuations to be isomorphic. Although many other algebraic aspects of these valuations are rather nontrivial and worthy of a more extensive study, here we only settle down the algebraic considerations that we will need in order to investigate their additive atomic structure.

\begin{lemma} \label{lem:negative generators}
	For a nonzero algebraic number $\beta$, the following statements hold.
	\begin{enumerate}
		\item If $\alpha$ is algebraic and has no positive conjugates, then $\nn_0[\alpha] = \zz[\alpha]$.
		\smallskip
		
		\item If $\alpha$ is algebraic and has a positive conjugate $\beta$, then $\nn_0[\alpha] \cong \nn_0[\beta]$, and so the monoid valuation $\nn_0[\alpha]$ is reduced.
	\end{enumerate}
\end{lemma}

\begin{proof}
	(1) Suppose now that $\alpha$ has no positive conjugates. Since the polynomial $m_\alpha(X)$ has no positive roots, Theorem~\ref{thm:Curtiss' multipliers} guarantees the existence of a nonzero polynomial $\phi(X) \in \zz[X]$ such that $\phi(X)m_\alpha(X) = \sum_{i=0}^k c_iX^i \in \nn_0[X]$. We can assume, without loss of generality, that $\phi(0) \neq 0$, that is, $c_0 \neq 0$. Since $\alpha$ is a root of $\phi(X)m_\alpha(X)$, we see that $-c_0 = \sum_{i=1}^k c_i \alpha^i \in \nn_0[\alpha] \cap \zz_{< 0}$ and, therefore, $-1 = -c_0 + (c_0 - 1) \in \nn_0[\alpha]$. As a consequence, $\{ \pm \alpha^n : n \in \nn_0 \}$ is contained in $\nn_0[\alpha]$, from which we obtain that $\nn_0[\alpha] = \zz[\alpha]$.
	\smallskip
	
	(2) Let $\omega$ be a real conjugate of $\alpha$ over $\qq$ (not necessarily positive), and consider the polynomial $m(X) = p(X) - q(X) \in \zz[X]$, where $(p(X),q(X))$ is the minimal pair of $\omega$. Let $R$ be the integral domain $\zz[X]/I$, where~$I$ is the prime ideal generated by $m(X)$ in $\zz[X]$. The ring homomorphism $\zz[X] \to \rr$ given by the assignments $f \mapsto f(\omega)$ (for all $f \in \zz[X]$) induces a ring isomorphism $\varphi \colon R \to \zz[\omega]$, namely, $\varphi \colon f(X) + I \mapsto f(\omega)$. Now observe that the set $S(X)$ of cosets of $\zz[X]/I$ having a representative in $\nn_0[X]$ is a subsemiring of $R$ satisfying $\varphi(S(X)) = \nn_0[\omega]$. Note that $\varphi(S(X))$ does not depend on $\omega$ but only on $m(X)$. Hence $\nn_0[\alpha] \cong S(X) \cong \nn_0[\beta]$ as semirings, which implies that $\nn_0[\alpha] \cong \nn_0[\beta]$ as monoids.
\end{proof}

As the factorization-theoretical aspects of groups and free commutative monoids are trivial (both are UFMs), by virtue of Lemma~\ref{lem:negative generators} there is no loss of generality in restricting our attention to monoid valuations $\nn_0[\alpha]$, where $\alpha$ is a positive algebraic number. We shall do so from now on. The following example illustrates another application of Lemma~\ref{lem:negative generators}.

\begin{example}
	For $\alpha = 2 - \sqrt{2}$, whose minimal polynomial is $m_\alpha(X) = X^2 - 4X + 2$, we have verified in Example~\ref{ex:two initial examples}(1) that the algebraic monoid valuation $\nn_0[\alpha]$ is not an HFM. Proving that $\nn_0[\alpha]$ is not an FFM directly may not be that simple. However, observe that $\beta = 2 + \sqrt{2}$ is also a root of $m_\alpha(X)$, that is, $\beta$ is a positive conjugate of $\alpha$. Thus, it follows from Lemma~\ref{lem:negative generators}(2) that $\nn_0[\alpha] \cong \nn_0[\beta]$. Now since $\beta > 1$, the generating sequence $(\beta^n)_{n \in \nn_0}$ of $\nn_0[\beta]$ is increasing and, therefore, \cite[Theorem~5.6]{fG19} guarantees that $\nn_0[\beta]$ is an FFM. Hence we can conclude that the monoid $\nn_0[\alpha]$ is an FFM that is not an HFM.
\end{example} 

The examples of algebraic monoid valuations that we have seen so far are FFMs. However, even rational monoid valuations may not be FFMs.

\begin{example} \label{ex:cyclic ratinoal semirings}
	Fix $q \in \qq \cap (0,1)$. If $q = \frac 1k$ for some $k \in \nn$, then one can readily check that the rational monoid valuation $\nn_0[q]$ contains no atoms, and so it is not even atomic. On the other hand, assume that $q \neq \frac 1k$ for any $k \in \nn$. It is shown in \cite[Theorem~6.2]{GG18} that $\nn_0[q]$ is atomic with $\mathcal{A}(\nn_0[q]) = \{q^n : n \in \nn_0\}$. However, since $\mathsf{d}(q) q^n = (\mathsf{d}(q) - \mathsf{n}(q))q^n + \mathsf{d}(q) q^{n+1}$ for every $n \in \nn_0$, the sequence of principal ideals $(\mathsf{d}(q)q^n + \nn_0[q])_{n \in \nn_0}$ is ascending. Since the same sequence of principal ideals does not terminate, the monoid $\nn_0[q]$ does not satisfy the ACCP. In particular, $\nn_0[q]$ is not an FFM. We finally observe that the fact that $\nn_0[q]$ does not satisfy the ACCP can also be inferred as a consequence of Theorem~\ref{thm:ACCP necessary condition} (see Corollary~\ref{cor:rational cyclic semiring that are not ACCP}).
\end{example}

Let us find a formula for the rank of an algebraic monoid valuation.

\begin{prop} \label{prop:rank of CARs}
	For $\alpha \in \rr_{>0}$, the equality $\emph{gp}(\nn_0[\alpha]) = \{f(\alpha) : f(X) \in \zz[X]\}$ holds. In addition, if~$\alpha$ is algebraic, then $\emph{\rank} \, \nn_0[\alpha] = \deg m_\alpha(X)$.
\end{prop}

\begin{proof}
	Set $G := \{f(\alpha) : f(X) \in \zz[X]\}$. Since $G$ is an abelian group containing $\nn_0[\alpha]$, it follows that $\gp(\nn_0[\alpha]) \subseteq G$. Observe, on the other hand, that the generating set $\{\pm \alpha^n : n \in \nn_0\}$ of $G$ is contained in $\gp(\nn_0[\alpha])$. As a consequence, $G \subseteq \gp(\nn_0[\alpha])$.
	\smallskip
	
	Suppose now that $\alpha$ is algebraic, and write $m_\alpha(X) = X^d -\sum_{j=0}^{d-1} c_j X^j$ for $c_0, \dots, c_{d-1} \in \qq$. Consider $\nn_0[\alpha]$ as an additive submonoid of the $\qq$-vector space $\qq \otimes_\zz \gp(\nn_0[\alpha])$ via the embedding $\nn_0[\alpha]  \hookrightarrow \gp(\nn_0[\alpha]) \hookrightarrow \qq \otimes_\zz \gp(\nn_0[\alpha])$. Let $V$ denote the subspace of $\qq \otimes_\zz \gp(\nn_0[\alpha])$ spanned by the set $S := \{\alpha^j : j \in \ldb 0,d-1 \rdb \}$. Since $\alpha^{d+n} = \sum_{j=0}^{d-1} c_j \alpha^{j+n}$ for every $n \in \nn_0$, we can argue inductively that $\alpha^n \in V$ for every $n \in \nn_0$. Hence $V = \qq \otimes_\zz \gp(\nn_0[\alpha])$. Since $\alpha$ is an algebraic number of degree $d$, it immediately follows that $S$ is a linearly independent over $\qq$ and so a basis for $V$. Hence $\rank \, \nn_0[\alpha] = \dim V = d$.
\end{proof}

We conclude this section by determining the isomorphism classes of semiring valuations of $\nn_0[X]$ whose additive monoids are atomic.

\begin{prop} \label{prop:isomorphism classes of CARS}
	For a nonzero $\alpha \in \rr$, suppose that the monoid $\nn_0[\alpha]$ is atomic. Then the following statements hold.
	\begin{enumerate}
		
		\item If $\alpha$ is rational, then for each $\beta \in \rr_{> 0}$ there exists a semiring isomorphism $\nn_0[\alpha] \cong \nn_0[\beta]$ if and only if $\nn_0[\alpha] = \nn_0[\beta] = \nn_0$ or $\alpha = \beta$.
		\smallskip
		
		\item If $\alpha$ is algebraic but not rational, then for each $\beta \in \rr_{> 0}$ there exists a semiring isomorphism $\nn_0[\alpha] \cong \nn_0[\beta]$ if and only if $\beta$ is an algebraic conjugate of $\alpha$.
	\end{enumerate}
\end{prop}

\begin{proof}
	(1) If $\alpha$ is rational, then $\rank \, \nn_0[\alpha] = 1$ by Proposition~\ref{prop:rank of CARs}. Suppose, for the direct implication, that $\varphi \colon \nn_0[\alpha] \to \nn_0[\beta]$ is a semiring isomorphism, and so $\nn_0[\alpha] \cong \nn_0[\beta]$ as semirings. Since $\rank \, \nn_0[\beta] = 1$, Proposition~\ref{prop:rank of CARs} ensures that $\beta$ is also rational. Then both $\nn_0[\alpha]$ and $\nn_0[\beta]$ are additive submonoids of $\qq_{\ge 0}$, and so \cite[Proposition~3.2]{fG18} guarantees the existence of $q \in \qq_{>0}$ such that $\varphi(x) = qx$ for all $x \in \nn_0[\alpha]$. Now, $\varphi(1) = 1$ implies that $q=1$, and so $\nn_0[\alpha] = \nn_0[\beta]$. Suppose that $\nn_0[\alpha] \neq \nn_0$. Then it follows from \cite[Proposition~4.3]{CGG20a} that either $\nn_0[\alpha] = \nn_0[\beta] = \nn_0$ or $\{\alpha^n : n \in \nn_0\} = \mathcal{A}(\nn_0[\alpha]) = \mathcal{A}(\nn_0[\beta]) = \{\beta^n : n \in \nn_0\}$.  In the later case, it is clear that $\alpha = \beta$. The reverse implication is straightforward.
	\smallskip
	
	(2) We will only prove the direct implication because the reverse implication follows the same lines as the proof of part~(2) of Lemma~\ref{lem:negative generators}. To begin with, we claim that $\alpha \in \mathcal{A}(\nn_0[\alpha])$. Suppose, otherwise, that $\alpha \notin \mathcal{A}(\nn_0[\alpha])$. Then there are elements $x,y \in \nn_0[\alpha] \setminus \{0\}$ such that $\alpha = x+y$, and so $\alpha^n = \alpha^{n-1}x + \alpha^{n-1}y$ for every $n \in \nn$, and so no positive power of $\alpha$ is an atom of $\nn_0[\alpha]$. Because $\mathcal{A}(\nn_0[\alpha]) \subseteq \{\alpha^n : n \in \nn_0\}$, this implies that $\mathcal{A}(\nn_0[\alpha]) \subseteq \{1\}$, which is not possible because $\nn_0[\alpha]$ is atomic. Let $\varphi \colon \nn_0[\alpha] \to \nn_0[\beta]$ be a semiring isomorphism for some $\beta \in \rr_{>0}$. Since $\alpha$ is algebraic but not rational, one can use Proposition~\ref{prop:rank of CARs} to deduce that $\beta$ is also algebraic but not rational. We have already argue that $\beta \in \mathcal{A}(\nn_0[\beta])$. Now take $f(X) \in \nn_0[X]$ such that $\varphi(f(\alpha)) = \beta$. As $\beta \in \mathcal{A}(\nn_0[\beta])$ and $\beta = f(\varphi(\alpha))$, it follows that $f(X)$ is a monic monomial, and so $\beta = \varphi(\alpha^k)$ for some $k \in \nn$. Also, if $g(X) \in \nn_0[X]$ satisfies $\varphi(\alpha) = g(\beta)$, then $\varphi(\alpha) = g(\beta) = g(\varphi(\alpha^k)) = \varphi (g(\alpha^k))$, and so $\alpha = g(\alpha^k)$. Because $\alpha \in \mathcal{A}(\nn_0[\alpha])$, one finds that $g(X)$ must be a monic monomial, which implies that $\alpha = \alpha^{kn}$ for some $n \in \nn$. As a result, $k = n = 1$, and so $\beta = \varphi(\alpha)$. Since $\varphi$ is, in particular, a monoid isomorphism between $\nn_0[\alpha]$ and $\nn_0[\beta]$, it uniquely extends to a group isomorphism $\varphi \colon \gp(\nn_0[\alpha]) \to \gp(\nn_0[\beta])$ (also denoted by~$\varphi$). Thus, $m_\alpha(\beta) = m_\alpha(\varphi(\alpha)) = \varphi(m_\alpha(\alpha)) = \varphi(0) = 0$, which implies that $m_\alpha(X)$ is also the minimal polynomial of $\beta$.
\end{proof}

\bigskip
\section{Atomicity and the ACCP}
\label{sec:additive structure}

In this section, we begin to study the atomic structure of monoid valuations of $\nn_0[X]$ with a focus on the properties of being atomic and satisfying the ACCP.

\smallskip
\subsection{Atomicity}

To begin with, we determine the values of algebraic numbers $\alpha \in \rr_{> 0}$ for which $\nn_0[\alpha]$ is atomic, and we describe its set of atoms. It is clear that for any positive algebraic number $\alpha$, the inclusion $\mathcal{A}(\nn_0[\alpha]) \subseteq \{\alpha^n : n \in \nn_0\}$ holds. However, the reverse inclusion does not always hold; for instance, we have seen in Example~\ref{ex:cyclic ratinoal semirings} that $\mathcal{A}(\nn_0[\frac 1k])$ is empty for every $k \in \nn_{\ge 2}$. In addition, the following example shows a positive algebraic number $\alpha$ such that $1 \notin \mathcal{A}(\nn_0[\alpha])$, resulting in $\mathcal{A}(\nn_0[\alpha])$ being empty.

\begin{example}
	Consider the positive algebraic number $\alpha = \frac{\sqrt{5} - 1}{2}$, whose minimal polynomial is $m_\alpha(X) = X^2 + X - 1$. As $\alpha$ is a root of $m_\alpha(X)$, the identity $1 = \alpha^2 + \alpha$ holds. Therefore $1 \notin \mathcal{A}(\nn_0[\alpha])$. After multiplying the previous identity by $\alpha^n$ (for any $n \in \nn_0$), we see that $\alpha^n \notin \nn_0[\alpha]$ for any $n \in \nn$. Hence $\nn_0[\alpha]$ is an antimatter monoid.
\end{example}

In the following theorem, we completely determine when an algebraic monoid valuation $\nn_0[\alpha]$ is atomic by noticing that $1 \notin \mathcal{A}(\nn_0[\alpha])$ is the only condition preventing $\nn_0[\alpha]$ from being atomic. In the same theorem, we explicitly describe the set of atoms of $\nn_0[\alpha]$. 

\begin{theorem} \label{thm:atomic characterization}
	 For each algebraic $\alpha \in \rr_{>0}$, the monoid $\nn_0[\alpha]$ is atomic if and only if $1 \in \mathcal{A}(\nn_0[\alpha])$, and $\nn_0[\alpha]$ is antimatter otherwise. Also, if $\nn_0[\alpha]$ is atomic, then there is a $\sigma \in \nn \cup \{\infty\}$ such that
	 \begin{equation} \label{eq:set of atoms}
	 		\mathcal{A}(\nn_0[\alpha]) = \{\alpha^n : n \in [0, \sigma) \cap \nn_0 \}.
	 \end{equation}
	 If $\nn_0[\alpha]$ is finitely generated (and so atomic), then $\sigma =  \min \{n \in \nn : \alpha^n \in \langle \alpha^j : j \in \ldb 0,n-1 \rdb \rangle \}$.
\end{theorem}

\begin{proof}
	First, we observe that if $\alpha^k \notin \mathcal{A}(\nn_0[\alpha])$ for some $k \in \nn_0$, then $\alpha^{k+1} \notin \mathcal{A}(\nn_0[\alpha])$. Indeed, every sum decomposition of $\alpha^k$ of the form $\alpha^k = \sum_{i=0}^m c_i \alpha^i$ with coefficients $c_0, \dots, c_m \in \nn_0$ yields the sum decomposition $\alpha^{k+n} = \sum_{i=0}^m c_i \alpha^{i+n}$ of $\alpha^{k+n}$ for every $n \in \nn$. Therefore if $1 \notin \mathcal{A}(\nn_0[\alpha])$, then the set of atoms of $\nn_0[\alpha]$ is empty, and so $\nn_0[\alpha]$ is an antimatter monoid. In this case, $\nn_0[\alpha]$ cannot be atomic, from which the direct implication of the first statement follows.
	\smallskip
	
	Conversely, suppose that $1 \in \mathcal{A}(\nn_0[\alpha])$. Notice first that if $\alpha \ge 1$, then for every $n \in \nn$ there are only finitely many elements of $\nn_0[\alpha]$ in the interval $[0,n]$ and, therefore, the set $\nn_0[\alpha]$ can be listed increasingly. In this case, $\nn_0[\alpha]$ is atomic by~\cite[Theorem~5.6]{fG19}. Thus, one can assume that $\alpha < 1$. Since $\alpha < 1$, it follows that $\alpha^i \nmid_{\nn_0[\alpha]} \alpha^j$ when $i < j$. This, along with the fact that $1 \in \mathcal{A}(\nn_0[\alpha])$, implies that $\alpha^n \in \mathcal{A}(\nn_0[\alpha])$ for every $n \in \nn_0$. Hence $\nn_0[\alpha]$ is an atomic monoid.
	\smallskip
	
	We have seen before that $\nn_0[\alpha]$ is antimatter if $1 \notin \mathcal{A}(\nn_0[\alpha])$, and it is clear that $1 \notin \nn_0[\alpha]$ when $\nn_0[\alpha]$ is antimatter. As a result, $\nn_0[\alpha]$ is either atomic or antimatter, which completes the proof of the first statement. 
	\smallskip
	
	To argue the second statement, assume that $\nn_0[\alpha]$ is atomic. Now let $\sigma$ be the smallest $n \in \nn \cup \{\infty\}$ such that $\alpha^n \in \langle \alpha^j : j \in \ldb 0,n-1 \rdb \rangle$. We split the rest of the proof into the following two cases.
	\smallskip
	
	\noindent \emph{Case 1:} $\sigma < \infty$. In this case, the inclusion $\alpha^{\sigma} \in \langle \alpha^n : n \in \ldb 0, \sigma-1 \rdb \rangle$ holds, which implies that $\alpha \ge 1$. Then it follows from our initial observation that $\alpha^n \notin \mathcal{A}(\nn_0[\alpha])$ for any $n \ge \sigma$. Now suppose that $\alpha^m = \sum_{i=0}^k c_i \alpha^i$ for $m < \sigma$ and for some $c_0, \dots, c_k \in \nn_0$ with $c_k > 0$. Since $\alpha \ge 1$, it follows that $k \le m$. However, $k < m$ would contradict the minimality of $\sigma$. Hence $k = m$, which implies that $\alpha^m$ is an atom of $\nn_0[\alpha]$. As a result, $\mathcal{A}(\nn_0[\alpha]) = \{\alpha^n : n \in \ldb 0, \sigma-1 \rdb\}$.
	\smallskip
	
	\noindent \emph{Case 2:} $\sigma = \infty$. In this case, $\alpha \neq 1$. If $\alpha > 1$, then $\alpha^n \nmid_{\nn_0[\alpha]} \alpha^m$ when $n > m$, which implies that $\alpha^m \in \mathcal{A}(\nn_0[\alpha])$ if and only if $\alpha^m \notin \langle \alpha^n : n \in \ldb 0, m-1 \rdb \rangle$. Thus, $\mathcal{A}(\nn_0[\alpha]) = \{\alpha^n : n \in \nn_0\}$. On the other hand, assume that $\alpha < 1$ and, therefore, that $0$ is a limit point of $\nn_0[\alpha] \setminus \{0\}$. Then $\mathcal{A}(\nn_0[\alpha]) = \{\alpha^n : n \in \nn_0\}$ as, otherwise, our initial observation would imply that $\nn_0[\alpha]$ is finitely generated, contradicting that $0$ is a limit point of $\nn_0[\alpha] \setminus \{0\}$.
	\smallskip
	
	When $\nn_0[\alpha]$ is finitely generated, $\mathcal{A}(\nn_0[\alpha])$ is finite, and so the fact that $\mathcal{A}(\nn_0) \subseteq \{\alpha^n : n \in \nn_0\}$ ensures that $\sigma < \infty$. Hence the last statement follows from the argument given in Case~1.
\end{proof}

\begin{remark}
	With the notation as in Theorem~\ref{thm:atomic characterization}, when $\nn_0[\alpha]$ is atomic the fact that $\{\alpha^n : n < \sigma\}$ is a basis for the free commutative monoid $\mathsf{Z}(\nn_0[\alpha])$  allows us to naturally embed $\mathsf{Z}(\nn_0[\alpha])$ into $\nn_0[X]$. Thus, we can identify a factorization $z$ in $\mathsf{Z}(\nn_0[\alpha])$ with the polynomial $z(X)$ in $\nn_0[X]$ obtained after replacing $\alpha$ by $X$. We shall use this identification throughout the paper without explicit mention.
\end{remark}

As the property of being finitely generated is relevant in the context of this paper, we highlight the following characterization of the finitely generated monoids $\nn_0[\alpha]$, which is an immediate consequence of Theorem~\ref{thm:atomic characterization}.

\begin{cor}
	For $\alpha \in \rr_{>0}$, the monoid $\nn_0[\alpha]$ is an FGM if and only if there is an $n \in \nn_0$ such that $\mathcal{A}(\nn_0[\alpha]) = \{\alpha^j : j \in \ldb 0,n \rdb \}$.
\end{cor}

We proceed to provide the two sufficient conditions for an algebraic monoid valuation $\nn_0[\alpha]$ to be atomic in terms of the minimal polynomial of $\alpha$.

\begin{prop} \label{prop:atomicity sufficient conditions}
	Let $\alpha \in \rr_{>0}$ be an algebraic number with minimal polynomial $m_\alpha(X)$. Then the following statements hold.
	\begin{enumerate}
		\item If $\alpha \notin \qq$ and $|m_\alpha(0)| \neq 1$, then $\nn_0[\alpha]$ is atomic.
		\smallskip
		
		\item If $m_\alpha(X)$ has more than one positive root, then $\nn_0[\alpha]$ is atomic.
	\end{enumerate}
\end{prop}

\begin{proof}
	(1) Suppose, by way of contradiction, that $\nn_0[\alpha]$ is not atomic. It follows from Theorem~\ref{thm:atomic characterization} that $1 \notin \mathcal{A}(\nn_0[\alpha])$, and so there are $c_1, \dots, c_n \in \nn_0$ such that $1 = \sum_{i=1}^n c_i \alpha^i$. Then $\alpha$ is a root of the polynomial $f(X) := 1 - \sum_{i=1}^n c_iX^i \in \qq[X]$. As a result, we can write $f(X) = m_\alpha(X)g(X)$ for some $g(X) \in \qq[X]$, and it follows from Gauss' Lemma that $g(X) \in \zz[X]$. As $f(0) = 1$, one obtains that $|m_\alpha(0)| = 1$, which is a contradiction.
	\smallskip
	
	(2) Once again, assume towards a contradiction that the monoid $\nn_0[\alpha]$ is not atomic. As in the previous paragraph, we can write $1 = \sum_{i=1}^n c_i \alpha^i$ for some $c_1, \dots, c_n \in \nn_0$ and obtain a polynomial $f(X) = 1 - \sum_{i=1}^n c_iX^i$ having $\alpha$ as a root. Now Descartes' Rule of Signs guarantees that $\alpha$ is the only positive root of $f(X)$. Since $m_\alpha(X)$ divides $f(X)$ in $\qq[X]$, the roots of $m_\alpha(X)$ are also roots of $f(X)$. As a consequence, the only positive root of $m_\alpha(X)$ is $\alpha$, which is a contradiction.
\end{proof}

None of the sufficient conditions for atomicity offered in Proposition~\ref{prop:atomicity sufficient conditions} implies the other one. In addition, there are algebraic monoid valuations that are atomic and yet do not satisfy any of these two conditions. This is illustrated in the following examples.

\begin{example} \hfill
	\begin{enumerate}
		\item The polynomial $m(X) = X^2 - 4X + 1$ is irreducible and has two distinct positive roots, namely, $2 \pm \sqrt{3}$. In particular, it is the minimal polynomial of the positive non-rational number $2 + \sqrt{3}$. However, $|m(0)| = 1$.
		\smallskip
		
		\item Now consider the positive algebraic number $\alpha = \sqrt{3} - 1$. Since the minimal polynomial of $\alpha$ is $m_\alpha(X) = X^2 + 2X - 2$, the condition $|m_\alpha(0)| \neq 1$ holds. However, one can verify that $\alpha$ is the only positive root of $m_\alpha(X)$.
		\smallskip
		
		\item Let us find now an atomic algebraic valuation monoid that does not satisfy any of the sufficient conditions in Proposition~\ref{prop:atomicity sufficient conditions}. Consider the golden number $\alpha = \frac{1 + \sqrt{5}}2$, whose minimal polynomial is $m_\alpha(X) = X^2 - X - 1$. Since $\alpha > 1$, we see that $0$ is not a limit point of $\nn_0[\alpha] \setminus \{0\}$, and so $\nn_0[\alpha]$ is a BFM by \cite[Proposition ~4.5]{fG19}. In particular, $\nn_0[\alpha]$ is atomic. However, $|m_\alpha(0)| = 1$ and the other root of $m_\alpha(X)$, namely $\frac{1 -\sqrt{5}}2$, is negative.
	\end{enumerate}
\end{example}

\smallskip
\subsection{The ACCP}

A relevant class of atomic monoids is that of monoids satisfying the ACCP. In his study of Bezout rings, Cohn~\cite[Proposition~1.1]{pC68} asserted without giving a proof that the underlying multiplicative monoid of any integral domain satisfies the ACCP provided that it is atomic. As mentioned in the introduction, this was refuted in 1964 by Grams, who constructed in~\cite{aG74} a neat counterexample. As we shall see in this subsection, there are monoid valuations of $\nn_0[X]$ that are atomic but do not satisfy the ACCP. We proceed to offer a necessary condition for an algebraic monoid valuation of $\nn_0[X]$ to satisfy the ACCP.

\begin{theorem} \label{thm:ACCP necessary condition}
	Let $\alpha \in (0,1)$ be an algebraic number with minimal pair $(p(X),q(X))$. If $\nn_0[\alpha]$ satisfies the ACCP, then $p(X) - X^k q(X) \notin \nn_0[X]$ for any $k \in \nn_0$.
\end{theorem}

\begin{proof}
	Suppose that $\nn_0[\alpha]$ satisfies the ACCP, and assume towards a contradiction that there is a $k \in \nn_0$ such that $f(X) := p(X) - X^kq(X) \in \nn_0[X]$. Consider the sequence $\big( q(\alpha)\alpha^{nk} + \nn_0[\alpha] \big)_{n \in \nn}$ of principal ideals of $\nn_0[\alpha]$. Observe now that for every $n \in \nn$,
	\[
		q(\alpha)\alpha^{nk} = p(\alpha) \alpha^{nk} = \big( f(\alpha) + \alpha^k q(\alpha) \big) \alpha^{nk} = f(\alpha) \alpha^{nk} + q(\alpha) \alpha^{(n+1)k}.
	\]
	As a result, $q(\alpha)\alpha^{nk} \in q(\alpha) \alpha^{(n+1)k} + \nn_0[\alpha]$ for every $n \in \nn$, which means that the sequence of principal ideals $\big( q(\alpha) \alpha^{nk} + \nn_0[\alpha] \big)_{n \in \nn}$ is ascending. On the other hand, since $q(\alpha)\alpha^{nk}$ is the minimum of $q(\alpha)\alpha^{nk} + \nn_0[\alpha]$ for every $n \in \nn$ and the sequence $(q(\alpha) \alpha^{nk})_{n \in \nn}$ decreases to zero, the chain of ideals $\big( q(\alpha) \alpha^{nk} + \nn_0[\alpha] \big)_{n \in \nn}$ does not stabilize. This contradicts that $\nn_0[\alpha]$ satisfies the ACCP.
\end{proof}

From Theorem~\ref{thm:ACCP necessary condition}, we can easily deduce \cite[Corollary~4.4]{CGG20a}.

\begin{cor} \label{cor:rational cyclic semiring that are not ACCP}
	For each $q \in \qq \cap (0,1)$ with $\mathsf{n}(q) \ge 2$, the monoid $\nn_0[q]$ is atomic but does not satisfy the ACCP.
\end{cor}

\begin{proof}
	The monoid $\nn_0[q]$ is atomic by Theorem~\ref{thm:atomic characterization}. On the other hand, taking $k=1$ in Theorem~\ref{thm:ACCP necessary condition}, we can see that $\nn_0[q]$ does not satisfy the ACCP.
\end{proof}

As the following example illustrates, the necessary condition in Theorem~\ref{thm:ACCP necessary condition} is not sufficient to guarantee that an atomic algebraic monoid valuation satisfies the ACCP.

\begin{example}
	The polynomial $X^3 + X^2 + X - 2$ is strictly increasing in $\rr_{\ge 0}$ and, thus, it has exactly one positive root, namely, $\alpha$. Since $X^3 + X^2 + X - 2$ is irreducible, it is indeed the minimal polynomial $m_\alpha(X)$ of $\alpha$, and so the minimal pair of $\alpha$ is $(p(X), q(X)) = (X^3 + X^2 + X, 2)$. Since $|m_\alpha(0)| \neq 1$, Proposition~\ref{prop:atomicity sufficient conditions} guarantees that $\nn_0[\alpha]$ is atomic. In addition, it is clear that $p(X) - X^k q(X) \notin \nn_0[X]$ for any $k \in \nn_0$, which is the necessary condition of Theorem~\ref{thm:ACCP necessary condition}.
	
	We proceed to verify that $\nn_0[\alpha]$ does not satisfy the ACCP. To do this, for every $n \in \nn$ set $x_n = \alpha^{n+2} + 2\alpha^{n+1} + 3\alpha^n \in \nn_0[\alpha]$ and consider the sequence $(x_n + \nn_0[\alpha])_{n \in \nn}$ of principal ideals of $\nn_0[\alpha]$. For every $n \in \nn$, the equality $2\alpha^n = \alpha^{n+3} + \alpha^{n+2} + \alpha^{n+1}$ holds and, therefore,
	\[
		x_n = \big( \alpha^{n+2} + 2 \alpha^{n+1} + \alpha^n \big) + 2\alpha^n = \big( \alpha^{n+2} + 2 \alpha^{n+1} + \alpha^n \big) + \big( \alpha^{n+3} + \alpha^{n+2} + \alpha^{n+1} \big) = x_{n+1} + \alpha^n.
	\]
	As a result, $x_n + \nn_0[\alpha] \subseteq x_{n+1} + \nn_0[\alpha]$ for every $n \in \nn$, which means that $(x_n + \nn_0[\alpha])_{n \in \nn}$ is an ascending chain of principal ideals of $\nn_0[\alpha]$. As in the proof of Theorem~\ref{thm:ACCP necessary condition}, one can readily see that the chain of ideals $(x_n + \nn_0[\alpha])_{n \in \nn}$ does not stabilize. Hence $\nn_0[\alpha]$ does not satisfy the ACCP.
\end{example}

As an application of Theorem~\ref{thm:ACCP necessary condition}, we conclude this section providing, for each $d \in \nn$, an infinite class of atomic monoid valuations of $\nn_0[X]$ of rank $d$ that does not satisfy the ACCP.

\begin{prop} \label{prop:infinitely many atomic CARS without the ACCP}
	For each $d \in \nn$, there exist infinitely many non-isomorphic semiring valuations of $\nn_0[X]$ whose additive monoids have rank~$d$, are atomic, but do not satisfy the ACCP.
\end{prop}

\begin{proof}
	Fix $d \in \nn$. Take $q \in \qq \cap (0,1)$ such that $\mathsf{n}(q) \ge 2$ is a squarefree integer, and consider the polynomial $m(X) = X^d - q \in \qq[X]$. The polynomial $m(X)$ is irreducible (by Eisenstein's Criterion) and has a root $\alpha_q$ in the interval $(0,1)$. Observe that the algebraic monoid valuation $\nn_0[\alpha_q]$ has rank~$d$ by Proposition~\ref{prop:rank of CARs}. If $d=1$, then $\nn_0[\alpha_q] = \langle q^n : n \in \nn_0 \rangle$, and so it is atomic by \cite[Theorem~6.2]{GG18}. If $d > 1$, then $\alpha_q$ is not rational, and so $\nn_0[\alpha_q]$ is atomic by part~(1) of Proposition~\ref{prop:atomicity sufficient conditions}. On the other hand, it follows from Theorem~\ref{thm:ACCP necessary condition} that $\nn_0[\alpha_q]$ does not satisfy the ACCP. Varying the parameter~$q$, one can obtain an infinite class of semirings $\nn_0[\alpha_q]$ with atomic additive monoids of rank $d$ that do not satisfy the ACCP. Finally, note that the semiring valuations of $\nn_0[X]$ in this class are pairwise non-isomorphic by Proposition~\ref{prop:isomorphism classes of CARS}.
\end{proof}

\smallskip
\subsection{The Bounded and Finite Factorization Properties}

The primary purpose of this subsection is to prove that, for monoid valuations of $\nn_0[X]$, the finite factorization property (and so the bounded factorization property) is equivalent to the ACCP.

\begin{theorem} \label{thm:BFM/FFM equivalence}
	For an algebraic number $\alpha \in \rr_{> 0}$, the following statements are equivalent.
	\begin{enumerate}
		\item[(a)] $\nn_0[\alpha]$ is an FFM.
		\smallskip
		
		\item[(b)] $\nn_0[\alpha]$ is a BFM.
		\smallskip
		
		\item[(c)] $\nn_0[\alpha]$ satisfies the ACCP.
	\end{enumerate}
\end{theorem}

\begin{proof}
	(a) $\Rightarrow$ (b) $\Rightarrow$ (c): Each FFM is clearly a BFM, and each BFM satisfies the ACCP by \cite[Corollary~1.3.3]{GH06}.
	\smallskip
	
	(c) $\Rightarrow$ (a): Suppose that $\nn_0[\alpha]$ is not an FFM. If $\nn_0[\alpha]$ is not atomic, then it cannot satisfy the ACCP, and we are done. Assume, therefore, that $\nn_0[\alpha]$ is atomic. We let $m_\alpha(X)$ denote the minimal polynomial of~$\alpha$. We consider the following two cases.
	\smallskip
	
	\noindent \emph{Case 1:} $m_\alpha(X)$ has more than one positive root (counting repetitions). Since $\nn_0[\alpha]$ is not an FFM, it follows from  \cite[Proposition~2.7.8]{GH06} that $\nn_0[\alpha]$ is not an FGM. This, along with Theorem~\ref{thm:atomic characterization}, ensures that $\mathcal{A}(\nn_0[\alpha]) = \{\alpha^n : n \in \nn_0\}$. For $x \in \nn_0[\alpha]$ and $n \in \nn_0$, set
	\[
		\mathsf{Z}_n(x) := \{z \in \mathsf{Z}(x) : n \in \supp \, z(X)\}.
	\]
	
	\noindent {\it Claim.} For each $x \in \nn_0[\alpha]$, if $|\mathsf{Z}_n(x)| < \infty$ for every $n \in \nn_0$, then $|\mathsf{Z}(x)| < \infty$.
	\smallskip
	
	\noindent {\it Proof of Claim.} Take a nonzero element $x \in \nn_0[\alpha]$ such that $\mathsf{Z}_n(x)$ is a finite set for every $n \in \nn_0$. Suppose, by way of contradiction, that $|\mathsf{Z}(x)| = \infty$. Fix $z_0 \in \mathsf{Z}(x)$, and set $d = \deg z_0(X)$. As~$\mathsf{Z}_n(x)$ is finite for every $n \in \ldb 0,d \rdb$, there is a factorization $z \in \mathsf{Z}(x)$ such that $\min \, \supp \, z(X) > d$. Consider the polynomial $z(X) - z_0(X) \in \zz[X]$. Since $z$ and $z_0$ are factorizations of the same element, $\alpha$ is a root of $z(X) - z_0(X)$. In addition, because $\min \supp \, z(X) > d = \deg z_0(X)$, Descartes' Rule of Signs guarantees that $\alpha$ is indeed the only positive root of $z(X) - z_0(X)$. However, the fact that $m_\alpha(X)$ divides $z(X) - z_0(X)$ contradicts that $m_\alpha(X)$ has more than one positive root. As a result, the claim follows.
	\smallskip
	
	Because $\nn_0[\alpha]$ is not an FFM, there is an $x_0 \in \nn_0[\alpha]$ such that $|\mathsf{Z}(x_0)| = \infty$. By the established claim, we can choose an $n_1 \in \nn$ so that $\mathsf{Z}_{n_1}(x_0)$ is an infinite set. Therefore $x_1 = x_0 - \alpha^{n_1} \in \nn_0[\alpha]$ satisfies $|\mathsf{Z}(x_1)| = \infty$. Now suppose that for some $j \in \nn$ we have found $x_0, \dots, x_j \in \nn_0[\alpha]$ such that $|\mathsf{Z}(x_j)| = \infty$ and $x_{i-1} - x_i \in \nn_0[\alpha] \setminus \{0\}$ for every $i \in \ldb 1,j \rdb$. Because $|\mathsf{Z}(x_j)| = \infty$, the previous claim guarantees the existence of an $n_{j+1} \in \nn_0$ such that the set $\mathsf{Z}_{n_{j+1}}(x_j)$ is infinite. Then after setting $x_{j+1} = x_j - \alpha^{n_{j+1}}$, one finds that $|\mathsf{Z}(x_{j+1})| = \infty$. Thus, we have constructed a sequence $(x_n)_{n \in \nn_0}$ of elements in $\nn_0[\alpha]$ satisfying $x_{n-1} - x_n \in \nn_0[\alpha] \setminus \{0\}$ for every $n \in \nn$. This implies that $(x_n + \nn_0[\alpha])_{n \in \nn_0}$ is an ascending chain of principal ideals of $\nn_0[\alpha]$ that does not stabilize. Hence $\nn_0[\alpha]$ does not satisfy the ACCP.
	\smallskip
	
	\noindent \emph{Case 2:} $\alpha$ is the only positive root of $m_\alpha(X)$ (counting repetitions). As $\nn_0[\alpha]$ is not an FFM, it follows from \cite[Theorem~5.6]{fG19} that it cannot be increasingly generated. Thus, $\alpha \in (0,1)$. Since $m_\alpha(X)$ has only one positive root, Theorem~\ref{thm:Curtiss' multipliers} ensures the existence of $\phi(X) \in \nn_0[X]$ such that $\phi(X) m_\alpha(X)$ belongs to $\zz[X]$ and has exactly one variation of sign. Set $f(X) = \sum_{i=0}^s c_i X^i =  \phi(X) m_\alpha(X)$, where $c_0, \dots, c_s \in \zz$. We can assume, without loss of generality, that $\phi(1) > 0$. As $f(1) = \phi(1)m_\alpha(1) \ge 1$, there is a $k \in \ldb 0, s-1 \rdb$ such that $\sum_{i=0}^{k+1} c_i \ge 0$ and $\sum_{i=0}^j c_i < 0$ for every $j \in \ldb 0,k \rdb$. After setting $\textbf{1}_n(X) = \sum_{i=0}^n X^i$ for every $n \in \nn_{> s}$, one obtains that
	\[
		f(X)\textbf{1}_n(X) = \sum_{j=0}^k \bigg( \sum_{i=0}^j c_i \bigg) X^j + \sum_{j=k+1}^{s-1} \bigg( \sum_{i=0}^j c_i \bigg) X^j + \sum_{j=s}^n f(1)X^j + \sum_{j=n+1}^{n+s} \bigg( \sum_{i=j-n}^s c_i\bigg) X^j.
	\]
	Observe that the negative coefficients of $f(X) \textbf{1}_n(X)$ are precisely the coefficients of the terms of degree at most $k$. Since $f(\alpha) \textbf{1}_n(\alpha) = 0$,
	\begin{equation} \label{eq:generators of the principal ideals}
		x_n := \sum_{j=0}^k \bigg( \! \! -\sum_{i=0}^j c_i \bigg) \alpha^j - \sum_{j=s}^n f(1)\alpha^j = \sum_{j=k+1}^{s-1} \bigg( \sum_{i=0}^j c_i \bigg) \alpha^j + \sum_{j=n+1}^{n+s} \bigg( \sum_{i=j-n}^s c_i\bigg) \alpha^j \in \nn_0[\alpha].
	\end{equation}
	Now consider the sequence $(x_n + \nn_0[\alpha])_{n \in \nn_{> s}}$ of principal ideals of $\nn_0[\alpha]$. It follows from~\eqref{eq:generators of the principal ideals} that $x_n - x_{n+1} = f(1)\alpha^{n+1} \in \nn_0[\alpha]$ for every $n > s$. Thus, $(x_n + \nn_0[\alpha])_{n \in \nn_{> s}}$ is an ascending chain of principal ideals of $\nn_0[\alpha]$ that does not stabilize. As a consequence, $\nn_0[\alpha]$ does not satisfy the ACCP.
\end{proof}
\smallskip

It follows from \cite[Proposition~2.7.8]{GH06} that every FGM is an FFM. However, there are algebraic monoid valuations of $\nn_0[X]$ that are FFMs but not FGMs. The following simple example, which will be significantly extended in Proposition~\ref{prop:FFMs that are not f.g.}, illustrates this observation.

\begin{example}
	Take $q \in \qq_{\ge 1} \setminus \nn$, and consider the rational monoid valuation $\nn_0[q]$. Since $\nn_0[q]$ is generated by the increasing sequence $(q^n)_{n \in \nn_0}$, it follows from~\cite[Theorem~5.6]{fG19} that $\nn_0[q]$ is an FFM. On the other hand, $\mathsf{d}(q) > 1$ implies that $q^n \notin \langle q^j : j \in \ldb 0, n-1 \rdb \rangle$ for any $n \in \nn$. Hence Theorem~\ref{thm:atomic characterization} guarantees that $\mathcal{A}(\nn_0[q]) = \{q^n : n \in \nn_0\}$. As a result, $\nn_0[q]$ is not an FGM.
\end{example}
\smallskip

We conclude this section with Figure~\ref{fig:atomic classes of CASs up to FGMs}, which is a visual summary of the results we have established so far.

\begin{figure}[h]
	\begin{tikzcd}[cramped]
		\textbf{FGM} \arrow[r, Rightarrow] \arrow[red, r, Leftarrow, "/"{anchor=center,sloped}, shift left=1.5ex] & \big[ \textbf{FFM} \arrow[r, Leftrightarrow] & \textbf{BFM}  \arrow[r, Leftrightarrow] & \textbf{ACCP monoid} \big] \arrow[r, Rightarrow] \arrow[red, r, Leftarrow, "/"{anchor=center,sloped}, shift left=1.5ex] & \textbf{ATM}
	\end{tikzcd}
	\caption{The non-obvious implications in the diagram are the main results we have established so far on the class of algebraic monoid valuations of $\nn_0[X]$, which include counterexamples justifying the red marked arrows.}
	\label{fig:atomic classes of CASs up to FGMs}
\end{figure}

\bigskip
\section{Factoriality}
\label{sec:factoriality}

We will show in this section that the implication UFM $\Rightarrow$ FGM holds in the class of algebraic monoid valuations of $\nn_0[X]$. Our primary purpose in this section is to study the properties of being half-factorial and length-factorial and extend the implication UFM $\Rightarrow$ FGM to the following chain of implications: UFM $\Leftrightarrow$ HFM $\Rightarrow$ LFM $\Rightarrow$ FGM. This may come as a surprise since, in general, an HFM (resp., an LFM) may not be an FGM even in the class of torsion-free reduced atomic monoids. In addition, in the same class, there are HFMs that are not UFMs. The following examples shed some light upon these observations.

\begin{example}
	Consider the additive submonoid $M$ of $\zz^2$ generated by the set $\{(1,n) : n \in \zz\}$. It is clear that $M$ is a torsion-free reduced atomic monoid with $\mathcal{A}(M) = \{(1,n) : n \in \zz\}$, from which one can deduce that $M$ is an HFM but not a UFM. However, $M$ is not an FGM. Indeed,~$M$ is not even an FFM: for instance, the equalities $(2,0) = (1,-n) + (1,n)$ (for every $n \in \nn$) yield infinitely many factorizations of $(2,0)$ in~$M$.
\end{example}

\begin{example}
	Consider the additive monoid $\nn_0[X]$, which is the free monoid of rank~$\aleph_0$, and then let~$M$ be the submonoid of $\nn_0[X]$ generated by the set $A = \{2, X^j + 2 : j \in \nn_0 \}$. One can readily argue that~$M$ is atomic with $\mathcal{A}(M) = A$. As a consequence,  $M$ is not an FGM. To show that $M$ is an LFM, consider two factorizations
	\[
		z := 2b_0 + \sum_{i=0}^n c_i(X^i + 2) \quad \text{and} \quad z' := 2b'_0 + \sum_{i=0}^n c'_i(X^i + 2)
	\]
	of the same element in $M$ satisfying $|z| = |z'|$, that is, $b_0, b'_0, c_i, c'_i \in \nn_0$ for every $i \in \ldb 0,n \rdb$ and $b_0 + \sum_{i=0}^n c_i = b'_0 + \sum_{i=0}^n c'_i$. Since, for every $j \in \nn$, the atom $X^j + 2$ is the only atom of $M$ whose support, as a polynomial, contains $j$, we can conclude that $X^j + 2$ is a prime in~$M$. As a result, $c_j = c'_j$ for every $j \in \ldb 1, n \rdb$. Thus, the equality $2b_0 + 3 c_0 = 2b'_0 + 3c'_0$ holds because $z$ and $z'$ are factorizations of the same element, and the equality $b_0 + c_0 = b'_0 + c'_0$ holds because $|z| = |z'|$. Therefore $c_0 = c'_0$ and $b_0 = b'_0$, from which we obtain $z = z'$. Hence $M$ is an LFM.
\end{example}

As we mentioned in the introduction, the notion of length-factoriality was introduced and first investigated by Coykendall and Smith~\cite{CS11} in 2011 under the term `other-half-factoriality': they proved that the multiplicative monoid of an integral domain is an LFM if and only if it is a UFM (a shorter proof of the same result was given in~\cite[Theorem~2.3]{AA10}\footnote{Although \cite{AA10} was published before~\cite{CS11}, the first proof that factoriality and length-factoriality are equivalent conditions in the class of integral domains is that given in~\cite{CS11}.}). After that, the study of length-factoriality seemed to be dormant for almost a decade until the second author~\cite{fG20a,fG20}, Chapman et al.~\cite{CCGS21}, and, even more recently, Geroldinger and Zhong~\cite{GZ21} considered length-factoriality in the setting of commutative monoids. It is worth noticing that monoids with elements having multiple factorizations of the same length, which are examples of non-LFMs, were investigated in~\cite{CGLM11} and, more recently, in~\cite{GGM21}. Still, at this point it seems like there is no example in the factorization theory literature of an LFM that is not an FFM. This suggests the following question.

\begin{question} \label{quest:LFM}
	Is every LFM an FFM?
\end{question}

\smallskip
\subsection{Half-Factoriality} 

As mentioned in the introduction, for monoid valuations of $\nn_0[X]$ the property of being an HFM is equivalent to that of being a UFM. Now we prove this assertion and, when the generator is algebraic, we give further characterizations of these equivalent properties.

\begin{theorem} \label{thm:UFM characterization}
	For an algebraic number $\alpha \in \rr_{>0}$, the following statements hold.
	\begin{enumerate}

		\item If $\nn_0[\alpha]$ is a UFM, then it is finitely generated.
		\smallskip
		
		\item If $\alpha$ has algebraic degree $d$, minimal polynomial $m_\alpha(X)$, and minimal pair $(p(X),q(X))$, then the following conditions are equivalent.
		\smallskip
		
		\begin{enumerate}
			\item[(a)] $\nn_0[\alpha]$ is a UFM.
			\smallskip
			
			\item[(b)] $\nn_0[\alpha]$ is an HFM.
			\smallskip
			
			\item[(c)] $\deg m_\alpha(X) = |\mathcal{A}(\nn_0[\alpha])|$.
			\smallskip
			
			\item[(d)] $p(X) = X^d$ for some $d \in \nn$.
		\end{enumerate}	
	\end{enumerate}
\end{theorem}

\begin{proof}
	(1) Suppose that $\nn_0[\alpha]$ is not finitely generated. Let $(p(X),q(X))$ be the minimal pair of $\alpha$. It follows from Theorem~\ref{thm:atomic characterization} that $\mathcal{A}(\nn_0[\alpha]) = \{\alpha^n : n \in \nn_0\}$. As a result, $p(\alpha)$ and $q(\alpha)$ are two distinct factorizations in $\mathsf{Z}(\nn_0[\alpha])$ of the same element, and so $\nn_0[\alpha]$ is not a UFM.
	\smallskip
	
	(2) To argue this part, suppose that $\alpha$ has algebraic degree $d$, minimal polynomial $m_\alpha(X)$, and minimal pair $(p(X),q(X))$.
	
	(a) $\Rightarrow$ (b): This is clear.
	\smallskip
	
	(b) $\Rightarrow$ (c): Suppose that $\nn_0[\alpha]$ is an HFM. If $\alpha \in \qq$, then $\nn_0[\alpha]$ is an additive submonoid of $(\qq_{\ge 0},+)$, and \cite[Proposition~4.2]{fG20} ensures that $\nn_0[\alpha] \cong (\nn_0,+)$. Then $\alpha \in \nn$, from which condition~(c) follows immediately. So assume that $\alpha \notin \qq$ and, therefore, that $\deg m_\alpha(X) > 1$. As in Theorem~\ref{thm:atomic characterization}, set $\sigma = \min \{n \in \nn : \alpha^n \in \langle \alpha^j : j\in \ldb 0, n-1 \rdb \rangle \}$. Since $\nn_0[\alpha]$ is atomic, it follows from Theorem~\ref{thm:atomic characterization} that $\mathcal{A}(\nn_0[\alpha]) = \{\alpha^j : j \in \ldb 0, \sigma-1 \rdb\}$. It is clear that $\deg m_\alpha(X) \le \sigma$. Suppose, by way of contradiction, that $\deg m_\alpha(X) < \sigma$. In this case, both $p(\alpha)$ and $q(\alpha)$ are distinct factorizations in $\nn_0[\alpha]$ of the same element. This, along with the fact that $\nn_0[\alpha]$ is an HFM, implies that $m_\alpha(1) = 0$. However, this contradicts that $m_\alpha(X)$ is an irreducible polynomial in $\qq[X]$ of degree at least $2$. Hence $\deg m_\alpha(X) = \sigma = |\mathcal{A}(\nn_0[\alpha])|$.
	\smallskip
	
	(c) $\Rightarrow$ (d): Because $\mathcal{A}(\nn_0[\alpha])$ is nonempty, $\nn_0[\alpha]$ is an atomic monoid by Theorem~\ref{thm:atomic characterization}. Let $d$ be the degree of $m_\alpha(X)$. Since $\mathcal{A}(\nn_0[\alpha]) = \{\alpha^j : j \in \ldb 0,d-1 \rdb\}$, there are coefficients $c_0, \dots, c_{d-1} \in \nn_0$ such that $\alpha^d = \sum_{i=0}^{d-1} c_i \alpha^i$. Therefore $\alpha$ is a root of the polynomial $X^d - \sum_{i=0}^{d-1} c_i X^i$ and, as a consequence, $m_\alpha(X) = X^d - \sum_{i=0}^{d-1} c_i X^i$ by the uniqueness of the minimal polynomial. Thus, $p(X) = X^d$.
	\smallskip
	
	(d) $\Rightarrow$ (a): Let $\sigma$ be defined as in the proof of (b) $\Rightarrow$ (c) above. Since $p(X) = X^d$, it is clear that $\alpha \ge 1$ and, therefore, the monoid $\nn_0[\alpha]$ is atomic by \cite[Proposition~4.5]{fG19}. On the other hand, the fact that $d$ is the degree of the minimal polynomial of $\alpha$ implies that $\sigma = d$, and so $\mathcal{A}(\nn_0[\alpha]) = \{\alpha^n : n \in \ldb 0, d - 1 \rdb \}$ by Theorem~\ref{thm:atomic characterization}. Take two factorizations $z_1, z_2 \in \mathsf{Z}(\nn_0[\alpha])$ of the same element in $\nn_0[\alpha]$. Then $\max \{ \deg z_1(X), \deg z_2(X) \} < \deg m_\alpha(X)$ and $z_1(\alpha) = z_2(\alpha)$. Since $\deg (z_1 - z_2)(X) < \deg m_\alpha(X)$, the fact that $m_\alpha(X)$ divides $z_1(X) - z_2(X)$ in $\qq[X]$ forces the equality $z_1(X) = z_2(X)$, which implies that $z_1 = z_2$. Hence $\nn_0[\alpha]$ is a UFM, and so Theorem~\ref{thm:atomic characterization} ensures that $\mathcal{A}(\nn_0[\alpha]) = \{\alpha^j : j \in \ldb 0, d-1 \rdb \}$.
\end{proof}

As an immediate consequence of the characterization given in Theorem~\ref{thm:UFM characterization}, we obtain the following result.

\begin{cor} \label{cor:HFM implies FFM in CARS}
	If $\nn_0[\alpha]$ is a UFM (or an HFM) for a positive algebraic $\alpha$, then $\nn_0[\alpha]$ is an FGM, and so an FFM.
\end{cor}

The converse of Corollary~\ref{cor:HFM implies FFM in CARS} does not hold, that is, there are algebraic monoid valuations of $\nn_0[X]$ that are FFMs but not HFMs. Indeed, we have verified in Example~\ref{ex:two initial examples}(2) that for every $q \in \qq \setminus \zz$ with $q > 1$, the rational monoid valuation $\nn_0[q]$ is an FFM that is not an HFM. In the direction of Example~\ref{ex:two initial examples}(2), we will construct in Proposition~\ref{prop:FFMs that are not f.g.} an infinite class of non-isomorphic algebraic monoid valuations of $\nn_0[X]$ (of any possible rank) that are FFMs but not FGMs and, therefore, not HFMs by virtue of Corollary~\ref{cor:HFM implies FFM in CARS}. Unlike the case of rational monoid valuations of $\nn_0[X]$, we will see in Proposition~\ref{prop:infinitely many f.g. CARS that are not OHFMs} that there are infinitely many non-isomorphic algebraic monoid valuations of $\nn_0[X]$ that are FGMs. The following proposition gives a necessary condition for $\nn_0[\alpha]$ to be an FGM.

\begin{prop} \label{prop:finitely generated CARS necessary conditions}
		If $\nn_0[\alpha]$ is an FGM for some algebraic $\alpha \in \rr_{>0}$, then $m_\alpha(X) \in \zz[X]$ and its only positive root is $\alpha$ (counting multiplicity).
\end{prop}

\begin{proof}
	We first prove that the polynomial $m_\alpha(X)$ belongs to $\zz[X]$. Since $\nn_0[\alpha]$ is an FGM, it follows from \cite[Proposition~2.7.8]{GH06} that it is an FFM. In particular, $\nn_0[\alpha]$ is atomic, and it follows from Theorem~\ref{thm:atomic characterization} that $\mathcal{A}(\nn_0[\alpha]) = \{\alpha^j : j \in \ldb 0,n-1 \rdb \}$ for some $n \in \nn$. Thus, we can write $\alpha^n = \sum_{i=0}^{n-1} c_i \alpha^i$ for some $c_0, \dots, c_{n-1} \in \nn_0$. Since $\alpha$ is a root of the polynomial $h(X) := X^n - \sum_{i=0}^{n-1} c_i X^i \in \qq[X] $, the minimal polynomial $m_\alpha(X)$ of $\alpha$ divides $h(X)$ in $\qq[X]$. Then $h(X) = m_\alpha(X)f(X)$ for some $f(X) \in \qq[X]$. Since $h(X)$ has integer coefficients, it follows from Gauss' Lemma that both $m_\alpha(X)$ and $f(X)$ belong to $\zz[X]$.
	
	It only remains to check that $\alpha$ is the only positive root of $m_\alpha(X)$. By virtue of Descartes' Rule of Signs, the polynomial $h(X)$ defined in the previous paragraph has exactly one positive real root, which must be~$\alpha$. This, together with the fact that $m_\alpha(X)$ divides $h(X)$ in $\qq[X]$, guarantees that $\alpha$ is the only positive real root of $m_\alpha(X)$.
\end{proof}

The necessary condition for the monoid $\nn_0[\alpha]$ to be finitely generated in Proposition~\ref{prop:finitely generated CARS necessary conditions} is not a sufficient condition, as the following example shows.

\begin{example}
	Let $\alpha$ be the only positive root of the polynomial $m_\alpha(X) = X^2 + 2X - 2$, and suppose towards a contradiction that $\nn_0[\alpha]$ is an FGM. The monoid $\nn_0[\alpha]$ is certainly atomic by part~(1) of Proposition~\ref{prop:atomicity sufficient conditions}, and so it follows from Theorem~\ref{thm:atomic characterization} that $\alpha^n \in \langle \alpha^j : j \in \ldb 0, n-1 \rdb \rangle$ for some $n \in \nn$. As a result, $m_\alpha(X)$ must divide a polynomial $X^n + \sum_{i=0}^{n-1} c_iX^i \in \zz[X]$ with $c_0, \dots, c_{n-1} \in \zz_{\le 0}$ for every $i \in \ldb 0,n-1 \rdb$. Notice that $n \ge 3$. Take $a_0, \dots, a_{n-3} \in \qq$ such that
	\begin{equation} \label{eq:to compare coefficients}
		(X^2 + 2X - 2) \bigg( X^{n-2} + \sum_{i=0}^{n-3} a_i X^i \bigg) = X^n + \sum_{i=0}^{n-1} c_iX^i.
	\end{equation}
	Observe that $a_0 = - \frac{1}{2}c_0 > 0$ and $a_1 = \frac{1}{2}(2a_0 - c_1) > 0$. In addition, if $a_j > 0$ for every $j \in \ldb 0,k \rdb$ for some $k < n-3$, then one can compare the coefficients of $X^{k+1}$ in both sides of~\eqref{eq:to compare coefficients} to find that $a_{k+1} = \frac{1}{2}(2a_k + a_{k-1} - c_{k+1}) > 0$. Hence $a_0, \dots, a_{n-3}$ are all positive. Now, after comparing the coefficients of $X^{n-1}$ in both sides of~\eqref{eq:to compare coefficients}, one obtains that $c_{n-1} = 2 + a_{n-3} > 0$, which is a contradiction. Thus, $\nn_0[\alpha]$ cannot be an FGM.
\end{example}

We have seen in Example~\ref{ex:two initial examples}(2) rank-$1$ monoid valuations of $\nn_0[X]$ that are FFMs but not FGMs. We conclude this section constructing for every $d \in \nn$ an infinite class of algebraic semiring valuations of $\nn_0[X]$ whose additive monoids are rank-$d$ FFMs but not FGMs.

\begin{prop} \label{prop:FFMs that are not f.g.}
	For each $d \in \nn$, there exist infinitely many non-isomorphic algebraic semiring valuations of $\nn_0[X]$ whose additive monoids are rank-$d$ FFMs that are not FGMs.
\end{prop}

\begin{proof}
	Suppose first that $d=1$. For each $\alpha \in \qq_{\ge 1} \setminus \nn$, consider the semiring valuation $\nn_0[\alpha]$. As $\nn_0[\alpha] \subseteq \qq$, the monoid $\nn_0[\alpha]$ has rank $1$. Since $\nn_0[\alpha]$ is increasingly generated, it follows from~\cite[Theorem~5.6]{fG18} that $\nn_0[\alpha]$ is an FFM. In addition, it is not hard to verify that $\mathcal{A}(\nn_0[\alpha]) = \{\alpha^n : n \in \nn_0\}$ (see \cite[Proposition~4.3]{CGG20a}), and so $\nn_0[\alpha]$ is not an FGM. Lastly, part~(1) of Proposition~\ref{prop:isomorphism classes of CARS} guarantees that $\nn_0[\alpha] \not\cong \nn_0[\beta]$ as semirings for any $\beta \in \qq_{\ge 1} \setminus \nn_0$ with $\beta \neq \alpha$.
	
	Suppose now that $d \ge 2$. Take $p \in \pp_{\ge 5}$, and then consider the polynomial $m(X) = (p-2)X^d + X - p$. We observe that $m(X)$ cannot have any complex root $\rho$ inside the closed unit disc as, otherwise, $p = |(p-2) \rho^d + \rho| \le (p-2) |\rho|^d + |\rho| \le p-1$. To verify that $m(X)$ is irreducible in $\zz[X]$ suppose, by way of contradiction, that $m(X) = f(X) g(X)$ for some $f(X), g(X) \in \zz[X] \setminus \zz$. Hence either $f(0)$ or $g(0)$ divides $p$. Assume, without loss of generality, that $|f(0)| = 1$ and set $n = \deg f(X)$. After denoting the complex roots of $f(X)$ by $\rho_1, \dots, \rho_n$ and its leading coefficient by $c$, one finds that $|\rho_1 \cdots \rho_n| = |\frac 1c | \le 1$. Therefore there is a $j \in \ldb 1,n \rdb$ such that $|\rho_j| \le 1$. However, this contradicts that $\rho_j$ is also a root of $m(X)$. Thus, $m(X)$ is irreducible. Since $m(1) < 0$, the polynomial $m(X)$ has a root $\alpha_p \in \rr_{> 1}$. Now Gauss' Lemma guarantees that $m_{\alpha_p}(X) = X^d + \frac{1}{p-2} X - \frac{p}{p-2} \in \qq[X]$ is the minimal polynomial of $\alpha_p$.
	
	It follows from Proposition~\ref{prop:rank of CARs} that $\nn_0[\alpha_p]$ has rank $d$. As in the case of $d=1$, the fact that $\nn_0[\alpha_p]$ is increasingly generated ensures that it is an FFM. Because $m_{\alpha_p}(X) \notin \zz[X]$, it follows from Proposition~\ref{prop:finitely generated CARS necessary conditions} that $\nn_0[\alpha_p]$ is not an FGM. Then for each prime $p \in \pp_{\ge 5}$ we can consider the semiring valuation $\nn_0[\alpha_p]$ of $\nn_0[X]$ whose additive monoid is a rank-$d$ FFM that is not an FGM. In light of Proposition~\ref{prop:isomorphism classes of CARS}, distinct parameters $p \in \pp_{\ge 5}$ yield non-isomorphic semirings $\nn_0[\alpha_p]$.
\end{proof}

\smallskip
\subsection{Length-Factoriality}

Let us call an LFM \emph{proper} if it is not a UFM. It was proved in~\cite{CS11} that the multiplicative monoid of an integral domain is never a proper LFM. There are, however, algebraic semiring valuations of $\nn_0[X]$ whose additive monoids are proper LFMs. Let us proceed to characterize such semirings. 

\begin{theorem} \label{thm:OHFM characterization}
	For an algebraic number $\alpha \in \rr_{> 0}$, the following conditions are equivalent.
	\begin{enumerate}
		\item[(a)] $\nn_0[\alpha]$ is a proper LFM.
		\smallskip
		
		\item[(b)] $\mathcal{A}(\nn_0[\alpha]) = \{\alpha^j : j \in \ldb 0, \deg m_\alpha(X) \rdb \}$.
	\end{enumerate}
\end{theorem}

\begin{proof}
	(a) $\Rightarrow$ (b): Since $\nn_0[\alpha]$ is not a UFM, it follows from Theorem~\ref{thm:UFM characterization} that
	\[
		\deg m_\alpha(X) < \sigma := \min \{n \in \nn : \alpha^n \in \langle \alpha^j : j \in \ldb 0, n-1 \rdb \rangle \}.
	\]
	Let $(p(X),q(X))$ be the minimal pair of $\alpha$, and suppose that $\deg m_\alpha(X) + 1 < \sigma$. Consider the polynomials $z_1(X) := p(X) + Xq(X)$ and $z_2(X) := q(X) + Xp(X)$ of $\nn_0[X]$. Since the equality $z_2(X) - z_1(X) = (X-1)m_\alpha(X)$ holds and $\deg m_\alpha(X) + 1 < \sigma$, it follows that $z_1(\alpha)$ and $z_2(\alpha)$ are two distinct factorizations in $\mathsf{Z}(\nn_0[\alpha])$ of the same element. As a result, $|z_1(\alpha)| = z_1(1) = z_2(1) = |z_2(\alpha)|$ guarantees that $\nn_0[\alpha]$ is not an LFM. Hence if $\nn_0[\alpha]$ is a proper LFM, then $\deg m_\alpha(X) = \sigma - 1$, and so (b) follows from Theorem~\ref{thm:atomic characterization}.
	\smallskip
	
	(b) $\Rightarrow$ (a): As $1 \in \mathcal{A}(\nn_0[\alpha])$, the monoid $\nn_0[\alpha]$ is atomic by Theorem~\ref{thm:atomic characterization}. Let $z_1, z_2 \in \mathsf{Z}(\nn_0[\alpha])$ be two factorizations of the same element such that $|z_1| = |z_2|$. In order to show that $z_1 = z_2$, there is no loss of generality in assuming that $z_1$ and $z_2$ have no atoms in common. Because $\alpha$ is a root of the polynomial $z_2(X) - z_1(X) \in \zz[X]$, whose degree is at most $\deg m_\alpha(X)$, there is a $c \in \qq$ such that $z_1(X) - z_2(X) = c \, m_\alpha(X)$. Since $|z_1| = |z_2|$, we see that $1$ is a root of $z_2(X) - z_1(X)$. However, as $\nn_0[\alpha]$ is an FGM, it follows from Proposition~\ref{prop:finitely generated CARS necessary conditions} that $1$ is not a root of $m_\alpha(X)$. Hence $c = 0$, which implies that $z_1(\alpha) = z_2(\alpha)$. As a result, $\nn_0[\alpha]$ is an LFM. That $\nn_0[\alpha]$ is not a UFM is an immediate consequence of part~(2) of Theorem~\ref{thm:UFM characterization}.
\end{proof}

The following corollary is an immediate consequence of Theorem~\ref{thm:OHFM characterization}.

\begin{cor}[cf. Question~\ref{quest:LFM}] \label{cor:OHFM implies f.g.}
	If $\nn_0[\alpha]$ is a proper LFM for some $\alpha \in \rr_{> 0}$, then it is an FGM and, therefore, an FFM.
\end{cor}

For every $d \ge 3$, we proceed to identify algebraic semiring valuations whose additive monoids are proper LFMs of rank~$d$.

\begin{prop}
	For $d \ge 3$ there are infinitely many non-isomorphic algebraic semiring valuations of $\nn_0[X]$ whose additive monoids are proper LFMs of rank $d$.
\end{prop}

\begin{proof}
	Fix $d \in \nn_{\ge 3}$ and $p \in \pp$. Now consider the polynomial
	\[
		m(X) = X^d - pX^{d-1} + pX^{d-2} - \sum_{i=0}^{d-3} pX^i.
	\]
	The polynomial $m(X)$ is irreducible by Eisenstein's Criterion. Since $m(1) = 1 - (d-2)p < 0$, the polynomial $m(X)$ has a real root $\alpha_p > 1$. Let $(p(X),q(X))$ be the minimal pair of $\alpha_p$, and consider the semiring valuation $\nn_0[\alpha_p]$. The monoid $\nn_0[\alpha_p]$ has rank $d$ by Proposition~\ref{prop:rank of CARs}. On the other hand, it follows from part~(1) of Proposition~\ref{prop:atomicity sufficient conditions} that $\nn_0[\alpha_p]$ is atomic. In addition,
	\[
		m(X) (X+1) = X^{d+1} - (p-1)X^d - \bigg( \sum_{i=1}^{d-3} 2p X^i \bigg) - p
	\]
	implies that $\mathcal{A}(\nn_0[\alpha_p]) = \{\alpha_p^j : j \in \ldb 0, k \rdb \}$ for some $k \in \{d-1,d\}$. However, notice that if $k = d-1$, then Theorem~\ref{thm:UFM characterization} would force $p(X)$ to be a monomial, which is not the case. Therefore we obtain that $\mathcal{A}(\nn_0[\alpha_p]) = \{\alpha_p^j : j \in \ldb 0, d \rdb \}$, and it follows from Theorem~\ref{thm:OHFM characterization} that $\nn_0[\alpha_p]$ is a proper LFM. Finally, observe that by Proposition~\ref{prop:isomorphism classes of CARS} different choices of $p \in \pp$ yield non-isomorphic semiring valuations $\nn_0[\alpha_p]$ whose additive monoids satisfy the desired conditions.
\end{proof}

We have just seen that if an algebraic monoid valuation $\nn_0[\alpha]$ is an LFM, then it is an FGM. These two properties are equivalent when $\nn_0[\alpha]$ has rank $2$.

\begin{prop} \label{prop:UFM/OHFM/FGM are equivalent in rank at most 2}
	Let $\alpha \in \rr_{> 0}$ be an algebraic number such that $\nn_0[\alpha]$ has rank at most $2$. Then the following conditions are equivalent.
	\begin{enumerate}
		\item[(a)] $\nn_0[\alpha]$ is a UFM.
		\smallskip
		
		\item[(b)] $\nn_0[\alpha]$ is an LFM.
		\smallskip
		
		\item[(c)] $\nn_0[\alpha]$ is an FGM.		
	\end{enumerate}
\end{prop}

\begin{proof}
	(a) $\Rightarrow$ (b): This is clear.
	\smallskip
	
	(b) $\Rightarrow$ (c): This is an immediate consequence of Theorem~\ref{thm:UFM characterization} and Corollary~\ref{cor:OHFM implies f.g.}.
	\smallskip
	
	(c) $\Rightarrow$ (a): Assume first that $\nn_0[\alpha]$ is an FGM with rank $1$. It follows from Proposition~\ref{prop:rank of CARs} that $\nn_0[\alpha]$ is a rational monoid valuation of $\nn_0[X]$. Since $\nn_0[\alpha]$ is an FGM, \cite[Proposition~4.3]{CGG20a} guarantees that $\nn_0[\alpha] = \nn_0$ and, therefore, it is a  UFM.
	
	Now assume that $\nn_0[\alpha]$ is an FGM with rank $2$. It follows from Proposition~\ref{prop:rank of CARs} that the irreducible polynomial $m_\alpha(X)$ of $\alpha$ has degree $2$, and it follows from Proposition~\ref{prop:finitely generated CARS necessary conditions} that $m_\alpha(X)$ belongs to $\zz[X]$ and has $\alpha$ as its unique positive root (counting multiplicity). Write $m_\alpha(X) = X^2 + aX - b$ for some $a,b \in \zz$. As $m_\alpha(X)$ has a unique positive root, Descartes' Rule of Signs guarantees that $b > 0$. 
	
	Suppose, by way of contradiction, that $a > 0$. Since $\nn_0[\alpha]$ is an FGM, there is a polynomial $f(X) \in \zz[X]$ such that $m_\alpha(X)f(X)$ is monic and its only positive coefficient is its leading coefficient (see the proof of Proposition~\ref{prop:finitely generated CARS necessary conditions}). Assume that the polynomial $f(X)$ has the least degree possible. Now write $f(X) = X^k + \sum_{i=0}^{k-1} c_i X^i$ for $c_0, \dots, c_{k-1} \in \zz$ and
	\begin{equation} \label{eq:product of polynomials}
		m_\alpha(X)f(X) = \big( X^2 + aX - b \big) \bigg( X^k + \sum_{i=0}^{k-1} c_i X^i \bigg) = X^{k+2} + \sum_{i=0}^{k+1} d_i X^i
	\end{equation}
	for $d_0, \dots, d_{k+1} \in \zz_{\le 0}$. Since $a>0$, we see that $\deg f(X) \ge 1$. As $d_0 < 0$, we obtain from~\eqref{eq:product of polynomials} that $c_0 > 0$. Observe that $\deg f(X) \ge 2$ as, otherwise, $d_2 \le 0$ and~\eqref{eq:product of polynomials} would imply that $c_0 \le -a < 0$. As $d_1 \le 0$, we obtain from~\eqref{eq:product of polynomials} that $c_1 \ge \frac{a c_0}b > 0$. As before, $\deg f(X) \ge 3$; otherwise, $d_3 \le 0$ would imply that $c_1 \le -a < 0$. For each $j \in \ldb 2,k-1 \rdb$, one can compare coefficients in~\eqref{eq:product of polynomials} to find that $bc_j  \ge c_{j-2} + a c_{j-1}$. Now an immediate induction reveals that $c_j  \ge 0$ for every $j \in \ldb 0, k-1 \rdb$. However, comparing the coefficients of the terms of degree $k+1$ in~\eqref{eq:product of polynomials}, we see that $c_{k-1} \le -a < 0$, which is a contradiction. As a consequence, $a \le 0$. Hence $\nn_0[\alpha]$ is a UFM by Theorem~\ref{thm:UFM characterization}.
\end{proof}

It follows from Proposition~\ref{prop:UFM/OHFM/FGM are equivalent in rank at most 2} that if $\nn_0[\alpha]$ is a proper LFM, then its rank is at least $3$. For each rank $d \in \nn_{\ge 3}$, there are infinitely many non-isomorphic algebraic semiring valuations of $\nn_0[X]$ whose additive monoids are rank-$d$ FGMs that are not LFMs.

\begin{prop} \label{prop:infinitely many f.g. CARS that are not OHFMs}
	For each $d \in \nn_{\ge 4}$, there exist infinitely many non-isomorphic semiring valuations of $\nn_0[X]$ whose additive monoids are rank-$d$ FGMs that are not LFMs.
\end{prop}

\begin{proof}
	Fix $d \in \nn_{\ge 4}$. Take $p \in \pp$, and consider the polynomial
	\[
		m(X) = X^d - 3pX^{d-1} + 2pX^{d-2} - \sum_{i=0}^{d-3} pX^i
	\]
	of $\zz[X]$. As $m(0) = -p$, the polynomial $m(X)$ has a positive root $\alpha_p$. Since $m(X)$ is irreducible (by virtue of Eisenstein's Criterion), it must be the minimal polynomial of $\alpha_p$. The monoid valuation $\nn_0[\alpha_p]$ has rank $d$ by Proposition~\ref{prop:rank of CARs}. In addition, it follows from Proposition~\ref{prop:isomorphism classes of CARS} that distinct choices of the parameter $p \in \pp$ yield non-isomorphic semiring valuations $\nn_0[\alpha_p]$ of $\nn_0[X]$. So proving the proposition amounts to showing that $\nn_0[\alpha_p]$ is an FGM that is not an LFM.
	
	For simplicity, set $\alpha = \alpha_p$. The monoid $\nn_0[\alpha]$ is atomic by part~(1) of Proposition~\ref{prop:atomicity sufficient conditions}. To show that $\nn_0[\alpha]$ is an FGM, set $f(X) := m(X) (X^2 + X + 1) \in \zz[X]$. One can immediately verify that the only positive coefficient of $f(X)$ is its leading coefficient. Therefore $\mathcal{A}(\nn_0[\alpha]) \subseteq \{\alpha^n : n \in \ldb 0, d+1 \rdb \}$. Since $\nn_0[\alpha]$ is atomic, it must be an FGM. To show that $\nn_0[\alpha]$ is not an LFM, it suffices to consider for every $b \in \zz$ the polynomial $g_b(X) = m(X) (X + b)$ and observe that one of the coefficients $-p(3b-2)$ and $-p(1-2b)$ of $g_b(X)$ (corresponding to degrees $d-1$ and $d-2$, respectively) must be positive. Thus,  $\nn_0[\alpha]$ is not an LFM by Theorem~\ref{thm:OHFM characterization}.
\end{proof}

Motivated by our proof of Proposition~\ref{prop:infinitely many f.g. CARS that are not OHFMs}, we are inclined to believe that the following related conjecture is true.

\begin{conj}
	For every $n \in \nn$ there is an algebraic number $\alpha \in \rr_{> 0}$ such that $|\mathcal{A}(\nn_0[\alpha])| - \deg m_\alpha(X) = n$.
\end{conj}

We conclude the paper with Figure~\ref{fig:atomic classes of CASs from FGMs}, which is a visual summary of the results established in this section. Observe that if we put together the diagram in Figure~\ref{fig:atomic classes of CASs from FGMs} and that in Figure~\ref{fig:atomic classes of CASs up to FGMs}, then we obtain the diagram in Figure~\ref{fig:atomic classes of CASs}, which is the summarizing diagram of the main results we have established in this paper.

\begin{figure}[h]
	\begin{tikzcd} 
		\big[ \textbf{UFM} \arrow[r, Leftrightarrow] & \textbf{HFM} \big] \arrow[r, Rightarrow] \arrow[red, r, Leftarrow, "/"{anchor=center,sloped}, shift left=1.5ex] & \textbf{LFM} \arrow[blue, r, Rightarrow] \arrow[red, r, Leftarrow, "/"{anchor=center,sloped}, shift left=1.5ex] &\textbf{FGM}
	\end{tikzcd}
	\caption{The non-obvious implications in the diagram represent the main results we have established in this section, which include counterexamples for the red marked arrows. The last implication (in blue) holds for all algebraic monoid valuations of $\nn_0[X]$ and the rest of the implications hold for all monoid valuations of $\nn_0[X]$.}
	\label{fig:atomic classes of CASs from FGMs}
\end{figure}
\smallskip

\bigskip
\section*{Acknowledgments}

The authors would like to thank an anonymous referee for many suggestions that helped to improve the present paper. During this collaboration, the second author was supported by the NSF postdoctoral award DMS-1903069.
\medskip

\bigskip

\end{document}